\documentclass{amsart}
\usepackage[english]{babel}
\usepackage[utf8]{inputenc}

\usepackage{graphicx}
\usepackage{amsthm}
\usepackage{amssymb}
\usepackage{hyperref}

\usepackage{amsmath}

\def\cal{\mathcal}

\newtheorem{theorem}{Theorem}[section]
\newtheorem{proposition}{Proposition}[section]
\newtheorem{corollary}{Corollary}[section]
\newtheorem{lemma}{Lemma}[section]
\theoremstyle{remark}
\newtheorem{definition}{Definition}[section]
\newtheorem{question}{Question}[section]
\newtheorem{remark}{Remark}[section]
\newtheorem{notation}{Notation}[section]
\newtheorem{example}{Example}[]

\begin{document}

\title[On the visibility of the $+$achirality of alternating knots]{On the visibility of the $+$achirality of alternating knots}
\author[N. Ermotti, C. V. Quach Hongler, and C. Weber]{Nicola Ermotti, Cam Van Quach Hongler, and Claude Weber}

\begin{abstract}
This article is devoted to the study of prime alternating +achiral knots. In the case of arborescent knots, we prove in +AAA Visibility Theorem 5.1, that the symmetry is visible on a certain projection (not necessarily minimal) and that it is realised by a homeomorphism of order 4. In the general case (arborescent or not), if the prime alternating knot has no minimal projection on which +achirality is visible, we prove that the order of +achirality is necessarily equal to 4.
\end{abstract}

\maketitle
\tableofcontents

\section{Introduction}

In this article, the knots in $S^3$ and the projections in $S^2$ are assumed to be {\bf prime and oriented}.

Let $K$ be an alternating knot. It is well known that one can detect from a minimal projection $\Pi$ of $K$ many topological invariants (such as the genus and the crossing number (see for instance \cite{crow}, \cite{mu}) and many topological properties such as whether it is fibered or not (see for instance \cite {ga}). Therefore, it is natural to raise about {\bf achirality} (see Definition 2.1) the following

\begin{question}
 Is it possible to see the achirality of an alternating knot on one of its projections?
\end{question}

For $-$achirality, the answer is yes by \cite{erquwe}. In addition, there exists a minimal projection on which $-$achirality is visible.

For +achirality, if the knot is arborescent there is a projection on which +achirality is visible. However, the projection may not be minimal. More precisely, we prove the following theorem \ref{AAA}

\bigskip
\noindent{\bf Theorem 5.1 (+ AAA Visibility Theorem).} 
{\it Let $K \subset S^3$ be a prime, alternating, arborescent knot. Suppose further that $K$ is +achiral. Then there exists a projection $\Pi_K \subset S^2$ of $K$ (not necessarily minimal) and a diffeomorphism $\Phi: S^2 \longrightarrow S^2$ of order 4 such that:
\begin{itemize}
\item[1.] $\Phi$ preserves the orientation of $S^2$.
\item[2.] $\Phi$ preserves the orientation of the projection.
\item[3.] $\Phi(\Pi_K) = \widehat {\Pi}_K$ where $\widehat{\Pi}_K$ denotes the image of $\Pi_K$ by reflection through the projection plane.
\end{itemize}
}

An example of alternating, arborescent and +achiral knots (denoted +AAA knots) without minimal projection on which the +achirality is visible, is given by Dasbach-Hougardy in \cite{daho}. A detailed study of this family of knots is made below in \S 6. For each one, we can show, following our proof of +AAA Visibility Theorem 5.1, a non-minimal projection on which the +achirality is visible. The case of non-arborescent knots is more complex. However, we can prove the following theorem.

\medskip
\noindent {\bf  Order 4 Theorem 7.1.}
Let $K$ be an alternating +achitral knot without minimal  $+$achiral projection. Then the order of +achirality of $K$ is equal to~4.

\subsection{Organisation of the paper}

In \S2 we present our definition of visibility. It is based on twisted rotations in $\bf {R}^3$ and on a theorem of Feng Luo which states that a hyperbolic +achiral knot  
$K$ in $S^3$ is isotopic to a knot $K^*$ invariant by a twisted rotation of finite order. Roughly speaking, the +achirality is visible on a projection $\Pi$ if the reflection plane of the twisted rotation coincides with the projection plane.

In \S3 we explain how visibility according to this definition can be realised in the case of alternating knots. The starting point is Key Theorem 3.1 which is only a restatement of Menasco-Thisthlethwaite Flyping Theorem for the symmetry +achirality of alternating  knots. If $\Pi$ is a minimal projection of the +achiral knot $K$, there exists an orientation preserving homeomorphism of the projection plane which, with the flypes, transforms $\Pi$ into $\widehat {\Pi}$ where $\widehat {\Pi}$ designates the image of $\Pi$ by reflection through the projection plane.

If the flypes were not needed, we would be done. Indeed the homeomorphism of Key Theorem 3.1 is isotopic to a homeomorphism of finite order, which is then conjugate to a twisted rotation by Finite Order Lemma 3.1 and a theorem of Kerekjarto (\cite{coko}).

However, flypes are usually needed. Thus, the objective is to understand how flypes perform. For this,   we use a tool provided by Bonahon-Siebenmann which structures a knot projection. The salient points of their theory are recalled in the appendix. This involves dividing the projection plane into diagrams which can be of two kinds: twisted band diagrams or jewels. The partition of the plane is encoded by a tree called {\bf structure tree}. The {\bf +achirality isomorphism} (which is a composition of homeomorphisms and flypes) induces an automorphism on the structure tree. It is a fact that this automorphism has exactly one fixed point. There are two possibilities (Theorem 3.2): 
\begin{itemize}
\item[i)]The fixed point is an edge of the tree. It corresponds to a Haseman (Conway) circle invariant by the achirality isomorphism.
\item[ii)] The fixed point is a vertex, which necessarily corresponds to a jewel.
\end{itemize}

In \S4 we study the minimal projection for achiral arborescent alternating knots for which only case i) can occur. The fixed point of the automorphism corresponds to a Conway circle invariant by the achirality isomorphism. Proposition 4.1 describes the schematic types of minimal projections of such knots. It is already done in (Section 6 \cite{erquwe}) for the case of knots. However, the proof in the present paper is generalised to the case of links. It is based on the Murasugi decomposition of alternating links (\cite {quwe2}).

Finally, we have Theorem 4.1 which gives a necessary and sufficient condition for an arborescent alternating knot to be +achiral. The order of +achirality of a +AAA knot is~4.

Section \S5 deals with +AAA Visibility Theorem 5.1. The visibility of the +achirality of +AAA knots is proved by introducing the notion of depth for arborescent tangles (defined in {\S 5.1}). An outline of the proof is given in {\S 5.4} using a move called $\alpha$-move ({\S 5.3}) which provides the induction step.

In \S 6 we study in detail a family of knots that we call DH-knots (Definition~5.1). An example of this class of knots is described by Dasbach-Hougardy in  \cite{daho}. They are not -achiral but +achiral without  minimal projection on which +achirality is visible. Following our constructive proof, using the $\alpha$-move, we can exhibit for these knots, a non-minimal projection in which +achirality is visible.

In \S7, we prove Order 4 Theorem 7.1 which states that a +achiral alternating knot which has no minimal achiral projection, has its +achirality order equal to~4. By Theorem 5.1, the proof consists essentially in the analysis of the situation where the fixed point of the automorphism of the structure tree is a jewel.

 Appendix (\S8) recalls the canonical decomposition of the projection of a knot.

\subsection{Some comments on terminology}
(1) This article is written with the intention of being self-contained. 
We have therefore included full proofs of facts known but not always stated in the appropriate form. For instance, Finite Order Lemma 3.1 and a result on rational tangles. However, we advise the reader to return to \cite{erquwe} for more details if necessary.

\smallskip
\noindent (2) We emphasise that our point of view is strictly two-dimensional, since we consider projections of alternating links. There are no hidden Conway spheres as pointed out in \cite{thi}. Conway circles are necessarily boundaries of alternating diagrams and are called Haseman circles.

\smallskip
\noindent (3) For the same reason, our terminology ``arborescent" is based exclusively on the class of alternating knot projections. Thanks to Menasco-Thistlethwaite Flyping Theorem  (\cite{meth}), our definition is intrinsic, i.e. it only depends on the knot type. 

\smallskip
\noindent (4) A link projection is the image of a link in $S^3$ by a linear and generic projection (in the sense of Reidemeister) on $S^2$. The term ``diagram" will be used to designate a different object, defined in the appendix.

\subsubsection*{Acknowledgements} We are grateful to the referees for their careful reading of our manuscript. Their insightful comments were a great help in writing this version..

\section{On the visibility of the symmetries of a knot}

\subsection{The group of symmetries of a knot and twisted rotations\\ in 3-space}

Let $K \subset S^3$ be a knot in $S^3$. Let $\pi_0\mathrm{Diff}(S^3, K)$ be the group of isotopy classes of diffeomorphisms $f: S^3 \rightarrow S^3$ such that $f(K) = K$. By definition, it is {\bf the group of symmetries} of the knot $K$. 

It is well known that one can replace diffeomorphisms with homeomorphisms in the definition above. Isotopy classes produce isomorphic quotient groups.
\begin{definition} \
\begin{enumerate}
\item
A knot $K$ is {\bf $+$achiral} if there exists a {\bf mirror homeomorphism} $\Psi: (S^3, K) \rightarrow (S^3, K)$ which reverses the orientation of $S^3$ but preserves the orientation of $K$.

\item A knot $K$ is {\bf $-$achiral} if there exists a {\bf mirror homeomorphism} $\Psi: (S^3, K) \rightarrow (S^3, K)$ which reverses the orientation of both $S^3$ and~$K$.
\end{enumerate}
\end{definition}
\begin{theorem}[Feng Luo \cite{luo})]
Suppose that $K$ is hyperbolic. Then the subgroup $C_k$ of $\pi_0\mathrm{Diff}(S^3, K)$ of symmetries which preserve the knot orientation is a cyclic subgroup of order $k$.
\end{theorem}

Suppose that $K$ is +achiral. Then the order $k$ of $C_k$ is even. Let $\omega$ be a generator of this cyclic group. Then $\omega$ reverses the orientation of $S^3$. Write $k = 2^{\mu} u$ with $\mu \geq 1$ and $u$ odd. Then $\omega^u$ also reverses the orientation of $S^3$ and is of order $2^{\mu}$. 

\begin{definition}
The order of +achirality of $K$ is $2^{\mu}$.
\end{definition}

For hyperbolic knots, Feng Luo gives explicit diffeomorphisms of $S^3$ for each type of symmetry. In the case of +achirality, they are described in  (\cite{luo}, Corollary 5) and stated below as Theorem 2.2.

Let $m \geq 1$ be an integer. Consider a rotation $r_m$ in ${\bf R}^3$ with angle $2\pi / m$ and the reflection $R_e$ through a plane orthogonal to the axis of rotation.

\begin{definition}
A {\bf twisted rotation} $h_m$ of angle $2\pi / m$ is the composition $R_e \circ r_m$. 
\end{definition}
Let $n$ be the order of $h_m$. Two cases can occur:
\begin{enumerate} 
\item
$m$ is even: Then $n = m$.
\item $m$ is odd. Since $(h_m)^m = R_e$, the order $n$ of $h_m$ is $n=2m$. 
\end{enumerate} 
Feng Luo proves the following theorem, a beautiful achievement of the efforts of many topologists.

\begin{theorem}
(\cite{luo})
Let $K \subset S^3$ be a hyperbolic knot. Suppose that $K$ is +achiral. Then $K$ is isotopic to a knot $K^*$ which is  invariant by a twisted rotation $h_m$ of angle $2\pi / m$.
\end{theorem}

In fact, we have:

\begin{proposition}
The integer $m$ of $h_m$ is even.
\end{proposition}

\begin{proof} Suppose $m$ is odd. Since $(h_m)^m$ is equal to the reflection $R_e$, $K^*$ is invariant by $R_e$.
This implies that $K^*$ is a non-trivial connected sum. This contradicts the fact that $K^*$ is hyperbolic and therefore prime. 
\end{proof}
\begin{corollary} The order $n$ of the twisted rotation $h_m$ is equal to the order $m$ of the rotation.
\end{corollary}
Corollary 2.1 and some arguments presented in \cite{luo} provide proof of the following proposition. 
\begin{proposition}
The twisted rotation $h_m$ generates a cyclic group $C_m$ of order $m$ in $Diff(S^3 , K)$ and $C_m$ is sent isomorphically on the cyclic subgroup $C_k$ in $\pi_0Diff(S^3 , K)$.
\end{proposition}

Corollary 2.1 and Proposition 2.2 imply the useful following fact:

\medskip
\noindent {\bf Fact.} Since $m=k=2^{\mu}u$, we obtain the order of +achirality $2^\mu$ with $\mu \geq 1$ from the order $n=m$ of the twisted rotation.
\medskip

We thank the referees for drawing our attention to the fact that in the original definition of a twisted rotation, the integer $m$ could be odd. Fortunately, this case does not occur for +achiral hyperbolic knots. So, what might have been a small gap in Feng Luo's article does not matter.

\subsection{Visibility}
Let $K$ be a hyperbolic $\pm$achiral knot with a projection $\Pi$. 
We now define the {\bf achirality visibility} of a knot:

\begin{definition} 
A projection $\Pi$ of a $\pm$achiral knot $K$ on $S^2$ in the ambient space $\bf R^3$ makes {\bf achirality visible} if $\Pi$ is invariant by a twisted rotation such that the mirror plane of the reflection $R_e$ is the projection plane $S^2$.
\end{definition}

\begin{definition} 
A projection $\Pi$ of a $\pm$achiral knot $K$ on a 2-sphere $S^2$ in the ambient space $\bf R^3$ makes the {\bf achirality visible} if $\Pi$ is invariant by a twisted rotation such that the intersection of $S^2$ with the mirror plane of the reflection $R_e$ is a great circle.
\end{definition}

Although Definitions 2.4 and 2.5 are different, if a projection $\Pi$ on $S^2$ makes achirality visible in the sense of Definition 2.5, we can construct another diagram $\Pi'$ from $\Pi$ that makes achirality visible in the sense of Definition 2.4. 

For our purposes, we will use the notion of visibility of achirality from Definition 2.5.

\begin{definition} A projection of a $\pm$achiral knot $K$ which makes visible the\linebreak $\pm$achirality is called a {\bf $\pm$achiral projection} of $K$.
\end{definition}

\begin{figure}[h!]
\includegraphics[scale=0.15]{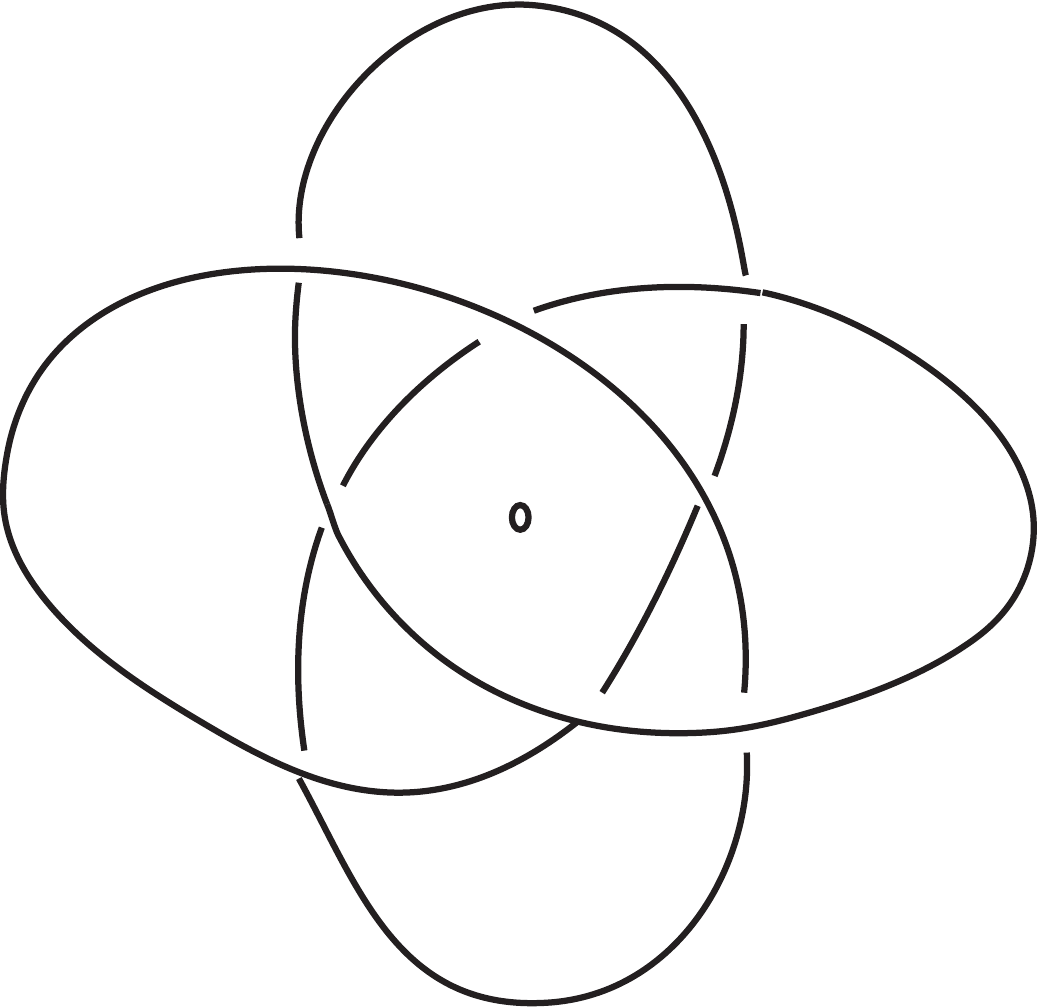}
\caption{Figure-eight knot.}
\end{figure}

\begin{example}
Fig.1 presents a +achiral projection of Figure-eight knot, as defined in Definition 2.4.
\end{example}

\section{Automorphism of the structure tree for an achiral knot}

\subsection{Key Theorem 3.1 and automorphism of the structure tree}
For basic facts about the canonical Bonahon-Siebenmann decomposition of a knot projection and about the structure tree, see Appendix. 

Fundamental to our analysis is {\bf Menasco-Thistlethwaite Flyping Theorem} which states that given any two reduced alternating diagrams $\Pi_1$ and $\Pi_2$ of an oriented prime alternating link $L$, $\Pi_1$ can be transformed to $\Pi_2$ by means of a finite sequence of flypes and orientation preserving homeomorphisms of $S^2$. The projections $\Pi_1$ and $\Pi_2$ are said to be {\bf isomorphic}.

Thus, our starting point is Key Theorem 3.1 which is only a reformulation of Menasco-Thisthlethwaite Flyping Theorem for $\pm$achiral alternating links.

\begin{theorem}[Key Theorem]
Let $K$ be a prime alternating oriented knot. Let $\Pi$ be an oriented minimal projection of $K$. Then $K$ is $\pm$achiral if and only if one can transform $\Pi$ into $\pm{\widehat {\Pi}}$ by a finite sequence of flypes and orientation preserving homeomorphisms of $S^2$.
\end{theorem}

\begin{figure}[h]
\includegraphics[scale=0.30]{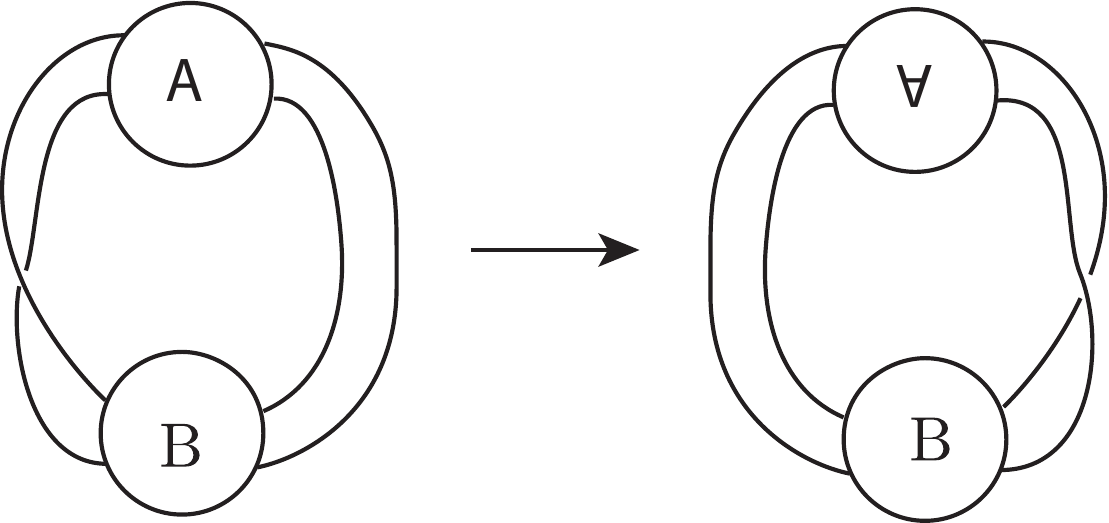}
\caption{A flype.}
\end{figure}

Let $K$ be an alternating achiral knot in $S^3$. Since the homeomorphisms and the flypes stated in Key Theorem 3.1 ``essentially commute", we can group the homeomorphisms together in a single homeomorphism called {\bf achirality homeomorphism}. Therefore, we can say that Theorem 3.1 asserts the existence of an isomorphism $\widetilde {\phi} : S^2  \longrightarrow S^2$ called {\bf achirality isomorphism} which is a composition of the achirality homeomorphism and a finite number of flypes and which is such that $\widetilde {\phi} (\Pi)=\widehat{\Pi}$.

Consider the structure trees ${\mathcal A}(K)$ and ${\mathcal A}(\widehat {K})$ (see \S 8.4). They only differ by the sign at the $\mathcal B$-vertices. As abstract trees without weights, they are canonically isomorphic.

Since this isomorphism induces an isomorphism ${\mathcal A}(K) \longrightarrow {\mathcal A}(\widehat{K})$, we can interpret it as an automorphism $\widetilde{\Phi}: {\mathcal A} (K) \longrightarrow {\mathcal A} (K)$ which sends a $\mathcal {B}$-vertex of weight $a$ to a $\mathcal {B}$-vertex of weight $-a$.

Lefschetz Fixed Point Theorem implies that $\widetilde{\Phi}$ has fixed points. We have proved in \cite{erquwe} the following result on the fixed point set of $\phi$.

\begin{theorem}
Let $K$ be a prime alternating achiral knot and $\widetilde{\Phi}: {\mathcal A} (K) \longrightarrow {\mathcal A} (K)$ be the automorphism induced by an achirality isomorphism of $K$. Then $\widetilde {\Phi}$ has exactly one fixed point.

There are two possibilities:
\begin{enumerate}
\item The fixed point is a vertex that corresponds to an invariant jewel.
\item The fixed point is the center of an edge whose two vertices are exchanged by $\widetilde {\Phi}$. Such an edge corresponds to a Haseman circle invariant by the achirality isomorphism and its vertices correspond both either to two twisted band diagrams or to two jewels.
\end{enumerate}
\end{theorem}

\subsection{Interpretation of Key Theorem 3.1 in terms of\\ twisted rotations}

Let $K$ again be an alternating +achiral knot and let $\Pi$ be a minimal projection of $K$. Then Theorem 3.1 implies that there exists a homeomorphism $\phi : (S^2 , \Pi) \longrightarrow (S^2 , \Pi)$ which preserves the orientation of both $S^2$ and $\Pi$ and which, after composition with the reflection $R_e$, realizes the +achirality of $K$, up to flypes.

2. Suppose no flype is needed. By Lemma 3.1 (see below in \S3.4) $\phi$ is isotopic, by an isotopy which leaves $\Pi$ invariant, to a finite order homeomorphism $ \phi^* : (S^2, \Pi) \longrightarrow (S^2 , \Pi)$. By Kerekjarto's theorem (\cite{coko}) $\phi^*$ is conjugate to a rotation, of course of finite order. So $R_e \circ \phi^*$ is conjugate to a twisted rotation which makes the + achirality visible.

3. Now suppose that some flypes are needed. Our approach is to reduce, as much as possible, the general case in case no flype is needed. The scheme of the reduction process is described as follows.

\subsection{Action of a homeomorphism on a decomposition of\\ the 2-sphere}

Suppose we have a finite decomposition $\mathcal{R} = \lbrace R_i \rbrace$ of the 2-sphere $S^2$ in connected planar surfaces, such that $R_i \cap R_j$ for $i \neq j$ is either empty or a common boundary component. Suppose further that we have a homeomorphism $g : S^2 \rightarrow S^2$ of finite order $n$ which respects the decomposition: for each index $i$, there exists an index $k(i)$ such that $g(R_i) = R_{k(i)}$. 

Consider some $R_i$ and the images $g(R_i) , g^2(R_i) , \dots , g^n(R_i) = R_i$. There are two possible cases:
\begin{itemize}
\item[i)] $g(R_i) , g^2(R_i) , \dots , g^n(R_i) = R_i$ are all distinct; we say that the orbit of $R_i$ is {\bf generic}.

\item[ii)] There exists an integer $m$ with $1 \leq m < n$ such that $g^m(R_i) = R_i$; we say that the orbit of $R_i$ is $\bf short$.
\end{itemize}

For a generic orbit, the restriction $g^n_ {\vert R_i}:R_i \rightarrow R_i$ is the identity. However for a short orbit, the restriction $g^m _{\vert R_i}:R_i \rightarrow R_i$ could be a non-trivial automorphism of $R_i$. In fact, the obstacles to an achiral projection come from the short orbits.

We now apply the above fact to a homeomorphism $\phi$ responsible of the $\pm$achirality symmetry of an alternating knot $K$. Let $\Sigma_i$ be a planar surface of the canonical decomposition of an alternating projection $\Pi$ of $K$. Let $\Gamma_i = \Pi \cap \Sigma_i$. Hence $(\Sigma_i , \Gamma_i)$ is a diagram of the decomposition. 

Let $\Sigma_j = \phi(\Sigma_i)$. Then $\phi(\Gamma_i) \subset \Sigma_j$ is flype equivalent in $\Sigma_j$ to $\Gamma_j$. If $i \neq j$, we transform $\Gamma_j$ by flypes into $\phi(\Gamma_i)$.

We continue these adjustments by flypes to $\phi^2(\Gamma_i)$ , \dots , $\phi^l(\Gamma_i)$ as long as it occurs in different diagrams. Now the following two cases can occur. 

i) The orbit of $\Sigma_i$ under $\phi$ is generic. We complete the adjustments when $l = n$. We do not need to make adjustments in the last step since $\phi^n \vert _{ \Sigma_i}$ is the identity. Therefore, the union of diagrams encountered in the sequence of modifications constitutes a subset of the projection $\Pi$, which is invariant by $\phi$.

ii) The orbit of $\Sigma_i$ is short. We put an end to the adjustments when $l = m$. But $\phi^m _{\vert \Sigma_i}: \Sigma_i  \rightarrow \Sigma_i$ could be a non-trivial automorphism of $\Sigma_i$. We know by hypothesis that $\Gamma_i$ is flype-equivalent to $\phi^m (\Gamma_i)$. But it is not certain that we can find a $\Gamma_i^*$ flype-equivalent to $\Gamma_i$ such that $\Gamma_i^* = \phi^m (\Gamma_i^*)$. If we can, we are done. If this is not possible, the involved knot has no minimal achiral projection. We will see in \S7 that this can only happen if the order of +achirality is equal to 4.

If this reduction process fails to produce a minimal projection on which +achirality is visible, there are two cases: 
\begin{itemize}
\item[1.] For arborescent knots, we can explicitly modify the projection in order to obtain a (non-minimal) projection that makes +achirality visible (\S5).
\item[2.] For non-arborescent knots, the case is open. 
\end{itemize}

\subsection {Finite order Lemma 3.1}

Although the following lemma is known by some specialists, we have not been able to find an explicit reference for the statement we need. Therefore, we include the lemma with proof for completeness. 
\begin{lemma}[Finite order lemma] Let $\Pi \subset S^2$ be a knot projection and let $\varphi : (S^2 , \Pi) \longrightarrow (S^2 , \Pi)$ be a homeomorphism. Then $\varphi$ is isotopic to a finite order homeomorphism (by an isotopy respecting $\Pi$).
\end{lemma}

\begin{proof}

Consider a projection $\Pi \subset S^2$. Associated with it is a cell decomposition of $S^2$ defined as follows:
\begin{itemize}
\item[0.] the 0-cells are the crossings.
\item[1.] the 1-cells are the arcs of $\Pi$ which connect two consecutive crossings.
\item[2.] the 2-cells are the closure of the regions of $S^2$ determined by $\Pi$.
\end{itemize}

The key to the proof is to triangulate the cell decomposition and then use the finitude properties of simplicial automorphisms. For this, we construct a simplicial complex $X$ as follows.

{\it0}. The set $X_0$ of vertices is a finite set in bijection with 0-cells. We denote such a bijection by $\eta_0$.
\begin{itemize}
\item[1.] For each 1-cell we add to $X_0$ a 1-simplex with extremities compatible with $\eta_0$. Let $X_1$ be the union of the 1-simplices. There is clearly a homeomorphism $\eta_1 : X_1 \longrightarrow \Pi$ which is restricted to $\eta_0$ on $X_0$. 
\item[2.] Let $C_2$ be a 2-cell  and let $\partial C_2$ be its boundary. The inverse image $(\eta_1)^{-1} (\partial C_2)$ is a cycle in the simplicial graph $X_1$.
One then forms a simplicial cone with the base this cycle and apex a new vertex. The cone contains 2-simplices and new 1-simplices. 
The union of $X_1$ and the cones obtained from all the 2-cells constitute the simplicial complex $X$.
Since each 2-cell is topologically a cone on its boundary (with apex some point in its interior), there is a homeomorphism $\eta : X \longrightarrow (S^2 , \Pi)$ which is restricted to $\eta_1$ on $X_1$. Admittedly, $\eta$ is not unique, but the choices involved do not have serious consequences. The homeomorphism $\eta$ triangulates the cell decomposition associated to $(S^2 , \Pi)$.
\end{itemize}
If in some particular cases, we do not get a simplicial complex of dimension 1 on $\Pi$, we can add some vertices to make it simplicial.

Now let $\varphi : (S^2 , \Pi) \longrightarrow (S^2 , \Pi)$ be the given homeomorphism. It induces a permutation of the 0-cells and of the 2-cells. Thus $\eta^{-1} \circ \varphi \circ \eta$ induces a permutation of the 0-simplices of $X_0$ and of the apexes of the cones. This permutation generates a simplicial automorphism of $\varphi : X \longrightarrow X$. Since this is a simplicial automorphism of a finite simplicial complex, it is of finite order. So $\eta \circ \varphi \circ \eta^{-1}$ is a homeomorphism $\varPhi$ of finite order of $(S^2 , \Pi)$.

Since $\varPhi$ and $\varphi$ both respect cell decomposition and induce the same permutation on  cells, they are isotopic by an isotopy which respects cells. To prove it, inductively use on the dimension of cells  Alexander's lemma which asserts the following fact: a homeomorphism of a closed cell which is the identity on its boundary is isotopic to identity by an isotopy fixed on the boundary.
\end {proof}

 \section[Minimal projection of an achiral alternating\\\mbox{} arborescent knot]{Minimal projection of an achiral alternating\\ arborescent knot}
The main goal of this section is to describe the possible minimal projections of the arborescent alternating achiral knots. However, many statements are valid for a more general framework (for projections corresponding to Case~2 in Theorem 3.2, i.e.  with an invariant Haseman circle).

Recall first that a tangle $\rm T$ is defined by a diagram (as described in Appendix) which is a pair $\rm{T}=(\Delta, \tau_{\Delta})$ where $\Delta$ is a disc on the projection plane $S^2$ and $\tau_{\Delta}$ is the intersection of a link projection $\Pi$ with $\Delta$. The {\bf boundary of the tangle} $\rm{T}$ denoted by $\partial \rm T$ is the boundary $\partial \Delta$ of the underlying disc $\Delta$. An {\bf orientation of the tangle} $\rm T=(\Delta, \tau_{\Delta})$ is an orientation of $\tau_{\Delta}$. 
\begin{definition} Two tangles $\rm{T}=(\Delta, \tau_{\Delta})$ and $\rm{T'}=(\Delta, \tau'_{\Delta})$ are {\bf isotopic} (denoted by $\rm T  \equiv \rm T'$) if there exists a homeomorphism $ f: \rm{T} \rightarrow \rm{T'}$ such that:
\begin{enumerate}
\item $f$ is the identity on the boundary $\partial \Delta$. 
\item $f(\tau_{\Delta})= \tau'_{\Delta}$.
\end{enumerate}

\end{definition}

\subsection{Splitting of the invariant Haseman circle}
Let $K$ be a prime achiral alternating arborescent knot and $\Pi$ a minimal projection of $K$ and let $\widetilde {\phi} : S^2  \longrightarrow S^2$ be an achirality isomorphism. Then, by Theorem 3.2 and Definition 8.10, there exists a unique Haseman circle $\gamma$ invariant by $\widetilde {\phi}$. Since $\Pi$ is arborescent, the two diagrams $D_1$ and $D_2$ adjacent to $\gamma$ which are swapped by the achirality isomorphism, are twisted band diagrams.
 
\begin{figure}[h!]
\includegraphics[scale=0.60]{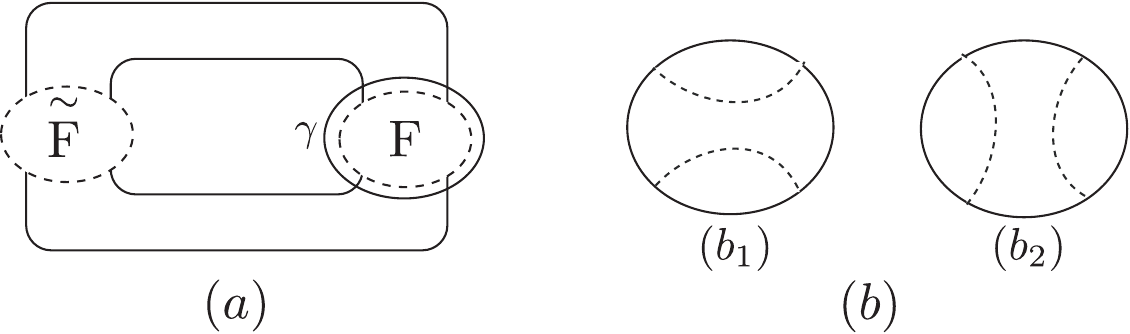}
\caption{(a) Splitting of the Haseman circle $\gamma$. 
(b) Horizontal and vertical Seifert arcs.}
\end{figure}

The circle $\gamma$ can be split into two circles and this gives rise to a partition of $\Pi$ into two tangles denoted by $\rm F$ and $\widetilde{\rm F}$ (Fig. 3). Next, our goal is to show that if $K$ is +achiral then $\Pi$ is isomorphic to a projection of type I (shown in Fig. 6 and Fig. 7). In (\cite{erquwe}, Section 6), this statement is proved by assuming that $K$ is a knot and not a link. In the following, we present another approach that is not only restricted to knots. It is based on the Murasugi decomposition of a minimal projection of an alternating link in atoms (recall from \cite{quwe2} that an {\bf atom} is a prime reduced alternating projection whose crossing number is equal to its writhe up to sign). The clue is to understand how Seifert circles behave in the Murasugi decomposition of the alternating projection $\Pi$. We have proven in (\cite{quwe2}), Theorem 5.2) that there is a Seifert circle $\Omega$ which passes through the two tangles $\rm F$ and $\widetilde{\rm F}$ such that $K$ is a diagrammatic Murasugi sum along $\Omega$ of a link $L$ and its mirror link $\pm \widehat{L}$; this is roughly denoted by $K= L \, * \, \pm \widehat{L}$. We call $\Omega$ the {\bf central Seifert circle} of $\Pi$.

(Recall that the mirror image of a link $L$ is obtained practically by changing all the crossings of a projection $\Pi_L$ of $L$. If the orientation of $\Pi_L$ is preserved by the achirality isomorphism $\widetilde {\phi}$, we denote the mirror image by $+ \widehat{L}$ or simply by $\widehat{L}$ and we say that $\widehat{L}$ and $L$ have the same orientation. If the orientation of $\Pi_L$ is reversed, we denote the mirror image by $- \widehat{L}$ and we say that $-\widehat{L}$ and $L$ have opposite orientations.)

\begin{lemma} 
The total number of Seifert circles of $\Pi$ is odd.
\end{lemma}

\begin{proof}
We prove the lemma by using the property that for an achiral alternating link, the number of positive atoms is equal to the number of negative atoms in its Murasugi decomposition (\cite{quwe2}).

Denote by $h( L)$ the first Betti number of the orientable surface obtained by the Seifert algorithm on any minimal projection of $L$. It is well known that  the Betti number $h(K)$, the crossing number $c(K)$ and the number of Seifert circles $s(K)$ are topological invariants of an alternating knot $K$ which are linked by the equation (see \cite{mu}, \cite{crow}):
\begin{equation}
h(K)= c(K) - s(K) + 1. 
\end{equation}

Let us denote the positive atoms by $L_i$ for $i =1,\dots, q$ and the negative atoms by $L_i$ for $i =q+1,\dots, 2q$. 

We have (see for instance \cite{qu}):
\begin{equation}
h(K) = \sum_{i=1}^{q} h(L_i) + \sum_{i=q+1}^{2q} h(L_i).
\end{equation}
In addition, the signature $ \sigma(K)$ of $K$ is 
\begin{equation}
 \sigma(K)= \sum_{i=1}^{q} h(L_i) - \sum_{i=q+1}^{2q} h(L_i).
 \end{equation}

Whatever the sign of the achirality of $K$, the signature $\sigma(K)$ is $0$. So (2) and (3) imply that $h(K)$ is even. Since $K$ is an achiral alternating knot, the crossing number $c(K)$ is even. So from (1), we can deduce that the number of Seifert circles $c(K)$ is odd.
\end{proof}

In the following, we denote by $\widehat{\rm F}$ the mirror image of $\rm F$ (obtained from $\rm F$ by changing all its crossings).
 
Neglecting orientations, every minimal projection of $K$ is isomorphic by  Key Theorem 3.1 to a projection $\Pi$ which is among the eight possible configurations shown in Fig. 4; the fact that $\rm F$ and $\widetilde{\rm F}$ are exchanged by the achirality isomorphism implies that keeping the right-hand sided tangle $\rm F$ fixed, a priori $\widetilde{\rm F}$ can have eight possible configurations of $\widehat{\rm F}$ as shown in Fig. 4.

We now identify what are the relevant configurations of $\widehat{\rm F}$ for $\pm$achiral alternating links.

The bunch of Seifert circles of $\Pi$ restricted to $\rm F$ contains some inner Seifert circles and two Seifert arcs. The two Seifert arcs shown in dotted lines in Fig. 3 are either horizontal $(b_1)$ or vertical $(b_2)$. These arcs belong either to one or two distinct Seifert circles of the projection $\Pi$.

\begin{figure}[h]
\includegraphics[scale=0.35]{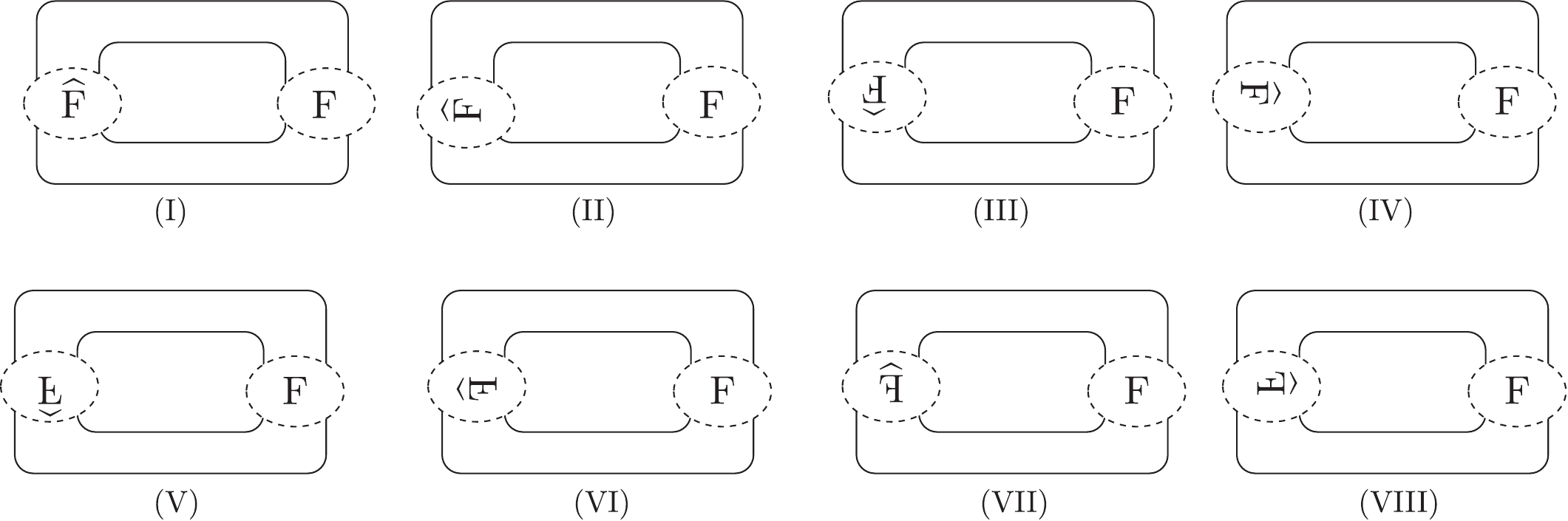}
\caption{The eight possible configurations.} 
\end{figure}

So we have:  
 \begin{lemma}$\,$ 
Assume without loss of generality that the two Seifert arcs of $\rm F$ are vertical. Then:
 \begin{enumerate}
 \item The Seifert arcs of \,$\widetilde {\rm F}$ are horizontal and the four strands connecting $\rm F$ to $\widetilde {\rm F}$  belong to the central Seifert circle.
  \item The configurations (I), (III), (V), and (VII) cannot occur.

  \end{enumerate}
\end{lemma}

\begin{proof}
{\it 1.} As $\rm F$ and $\widetilde {\rm F}$ are ``similar" by $\widetilde{\phi}$, these tangles have the same number of inner Seifert circles. If the Seifert arcs of $\widetilde {\rm F}$ were vertical, the total number of Seifert circles of $\Pi$ would be even, contradicting Lemma 4.1. Therefore, the Seifert arcs of $\widetilde {\rm F}$ are horizontal and the four strands connecting $\rm F$ to $\widetilde {\rm F}$ belong to the central Seifert circle $\Omega$ (Fig. 6).

{\it 2.} By {\it 1}, the configurations (I), (III), (V), and (VII) cannot occur because they correspond to the case where $\widetilde {\rm F}$ has its vertical Seifert arcs.
\end{proof}

\begin{figure}[h!]
\includegraphics[scale=0.4]{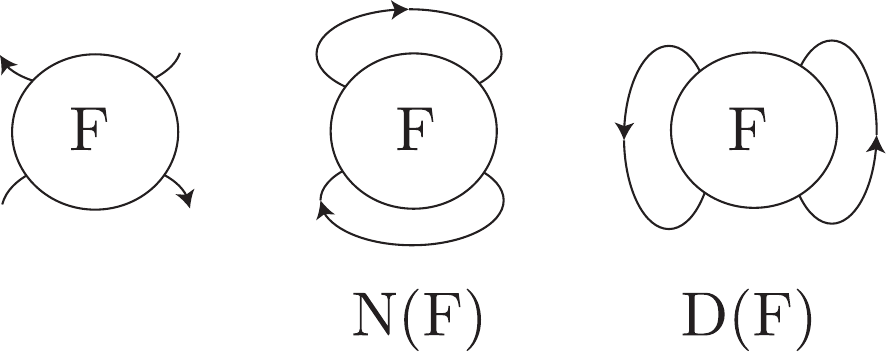}
\caption{Numerator $\rm N(\rm F)$ and Denominator $\rm D(\rm F)$.}
\end{figure}

\begin{remark}
Without loss of generality, suppose that the central Seifert circle $\Omega$ appears as in Fig. 6. Following this Seifert circle, we see that the orientations alternate on the boundary of $\rm F$.

The closures of $\rm F$ called as {\bf numerator} $\rm N(\rm F)$ and  {\bf denominator} $\rm D(\rm F)$ (\cite{crom}) are oriented according to the orientation of $\rm F$ (Fig. 5). So $\rm D(\rm F)$ and $ \rm N(\widetilde{\rm F})$ are respectively the link $L$ and its mirror link $\pm \widehat{L}$ in the Murasugi decomposition of $K= L \, * \, \pm \widehat{L}$ as shown in \cite{quwe2} .
\end{remark}

\subsection{Projections of Type I and of Type II}
\begin{figure}[h!]
\includegraphics[scale=0.55]{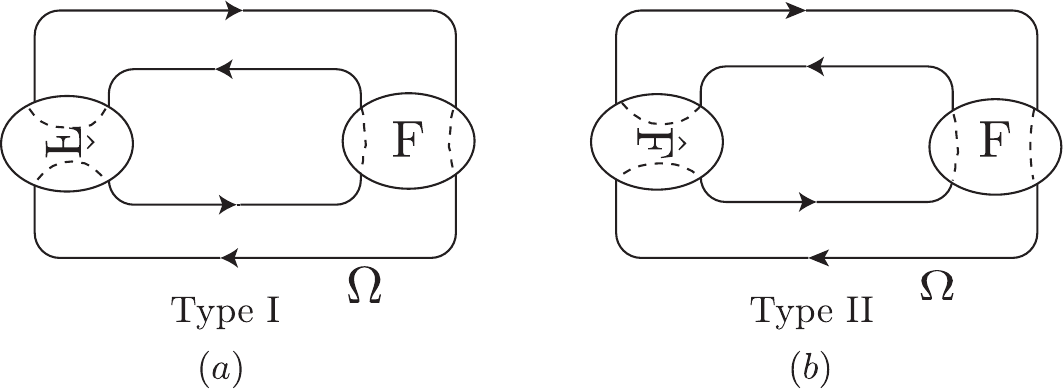}
\caption{The two types of projections.}
\end{figure}

According to Lemma 4.2, we only have 4 cases to take into consideration: (II), (IV), (VI) and (VIII). By a conjugation of the $\pi$-rotation along a horizontal line, the cases (II) and (IV) are equivalent. Likewise (VI) and (VIII) are equivalent by conjugation of the $\pi$-rotation along a vertical line. There are therefore only two cases left to deal with:
 
(VIII) called {\bf projection of type I} and  (VI) called {\bf projection of type II} (Fig. 6).
 
A minimal projection of an achiral alternating prime knot is isomorphic, up to a global change of orientation, to a minimal projection of type I or II (see Theorem 6.1 \cite{erquwe}).

\begin{proposition}$\,$
Let $K$ be a prime alternating knot and $\Pi$ a minimal projection with an invariant Haseman circle. Then:
\begin{enumerate} 
\item
If $K$ is $+$achiral, $\Pi$ is isomorphic to a projection of type I.

\item
If $\Pi$ is isomorphic to a projection of type II, $K$ is a $-$achiral knot with a Tait involution (i.e. an orientation preserving involution $\Phi$ of $S^2$ such that $\Phi(\Pi) = -\widehat {\Pi}$).

\end{enumerate}
\end{proposition}

\begin{figure}[h!]
\includegraphics[scale=0.5]{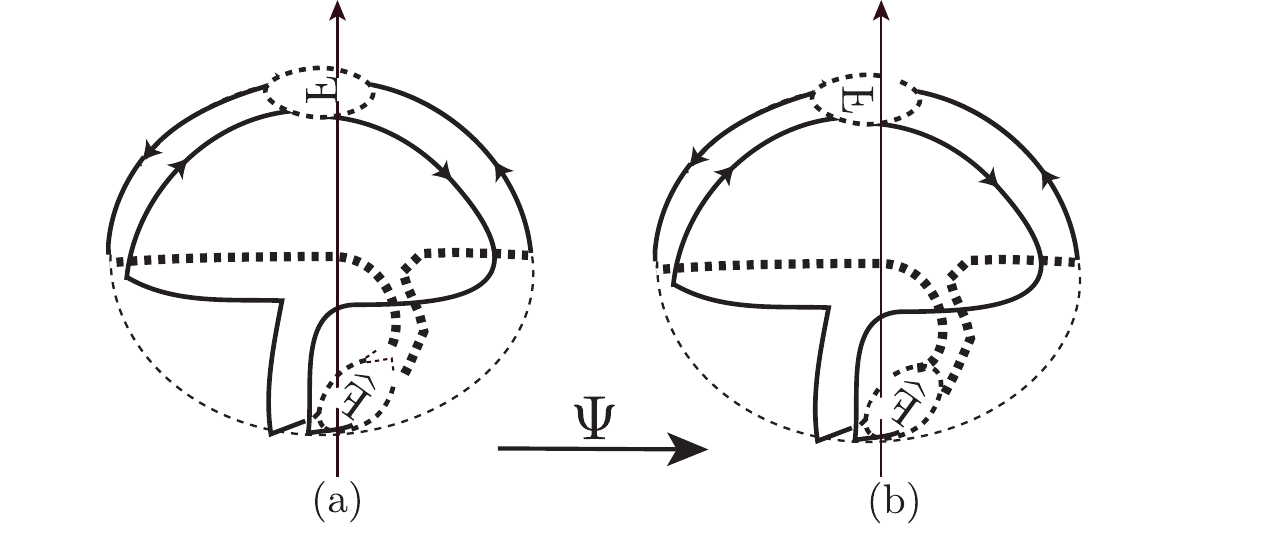}
\caption{(a) A projection of type I and (b) its image under a twisted rotation $\Psi$ of order 4.}
\end{figure}

\begin{proof}
It is easy to realise that in a projection of type I, the links $\rm D(\rm F)=L$ and $ \rm N(\widetilde{\rm F})=D(\widehat {F})=\widehat{L}$ have the same orientation and in a projection of type II, $ \rm D(\rm F)=L$ and  $\rm N(\widetilde{\rm F})=D(\widehat {F})= -\widehat{L}$ have opposite orientations.

$\it 1.$ By the Murasugi decomposition, since $K$ is +achiral , we have $K= L* \widehat{L}$ (\cite{quwe2}) for an alternating link $L$. So $\Pi$ is isomorphic to a projection of type I since $L=\rm D(\rm F)$ and its mirror image $ \widehat{L}=\rm N(\widetilde{\rm F})$ have the same orientation. 

$\it 2.$ A projection of type II corresponds to a $-$achiral knot. We have $ \rm D(\rm F)=L$ and $ \rm N(\widetilde{\rm F})=D(\widehat {F})= -\widehat{L}$ with opposite orientations; a Tait involution is visible.
\end{proof}

We introduce the following notation:
\begin{notation}
Let $\rm F$ be a tangle. 
\begin{enumerate}
\item $\rm F^*$ is  $\rm R^*(\rm F)$ where $\rm R^*$ is 
the half-turn rotation with center the ``middle" of $\rm F$.

\item ${\rm F}^v$ is the image of $\rm F$ under rotation of $\pi$ in the projection plane around its vertical axis. 
Similarly, ${\rm F}^h$ is the image of  ${\rm F}$ under rotation of $\pi$ around its horizontal axis .

\item ${\rm F} \sim {\rm G}$ if ${\rm F}$ and ${\rm G}$ are related by a homeomorphism which leaves the boundary circle fixed and by a finite sequence of flypes.
\item ${\rm F}$ is $*$-equivalent if $\rm F  \sim {\rm F}^*$.
\item ${\rm F}$ is $h$-equivalent (respectively $v$-equivalent) if ${\rm F}\sim {\rm F}^h$ (respectively ${\rm F}\sim {\rm F}^v$).
\item ${\rm F} = \rm F^*$ if ${\rm F}$ is invariant under $\rm R^*$ and we say that ${\rm F}$ is \bf{$*$-visible}.
\end{enumerate}
\end{notation}

\begin{definition}
Let  $\rm F$ be an alternating tangle such that $\rm F$ is $*$-equivalent. Then:
\begin{enumerate}
\item 
If there exists $\rm F'$ such that $\rm F' \sim \rm F$ (so $\rm F'$ is also alternating) and ${\rm F'}$ is $*$-visible, 
$\rm F$ is {\bf strictly $*$-visible}. 
\item
If there is no alternating tangle $\rm F' \sim \rm F$ such that ${\rm F'} \equiv {\rm F'}^*$, ${\rm F}$ is {\bf slightly $*$-visible}. 
\item ${\rm F}$ is a {\bf $*$-tangle} if it can be isotoped into a $*$-visible tangle.
\end{enumerate}
\end{definition}
\begin{example}  The tangle $\rm F$ shown in Fig. 8 is a $*$-tangle but it is only slightly $*$-visible (see Fig. 25).
\end{example}

\begin{figure}[h!]
\includegraphics[scale=0.45]{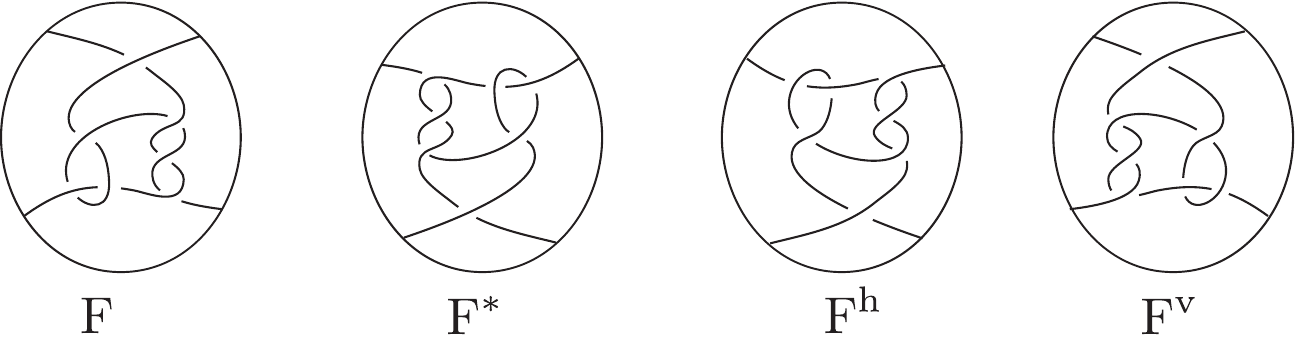}
\caption{ ${\rm F}$, ${\rm F}^*$, ${\rm F}^h$and ${\rm F}^v$.}
\end{figure}

\begin{theorem}[Theorem 6.5 \cite{erquwe}] 
Let $K$ be a $+$achiral alternating knot which is in case 2 of Theorem 3.2 and let $\Pi$ be a minimal alternating projection of $K$. Then $\Pi$ is isomorphic to a projection as shown in Fig.~6({\it a}) where its invariant Haseman circle $\gamma$ decomposes $\Pi$ into two tangles $\rm F$ and $\widehat {\rm F}$ and the order of $+$achirality of $K$ is $4$. In addition, ${\rm F} \sim \rm F^*$. 
\end{theorem}

\begin{proof}(\cite{erquwe} Proposition 6.3).
The analysis of the possible homeomorphisms of $S^2$ implies that the mirror homeomorphism $\Psi$ is a twisted rotation of order 4 (up to conjugation) whose reflection $R_e$ is by a large circle of $S^2$. The $+$achirality of $K$ implies that the condition $\rm F \sim F^*$ must be satisfied.
 \end{proof}

For a +AAA knot $K$ and $\Pi$ a minimal projection of $K$ described by Fig.~6({\it a}), we call the tangle $\rm F$ a {\bf primary tangle} of $\Pi$. 

\clearpage

\begin{remark}\ 
\begin{enumerate}
\item In the case where flypes are not needed, we have the invariance of $\rm F$ and ${\rm F}^*$ under $\rm R^*$ that is ${\rm F} = {\rm F}^*$ and $\widehat{\rm F} = \widehat {\rm F}^*$. Thus, the projection of type I as shown in Fig. 7 is invariant under a twisted rotation of order 4.

As an example, consider Figure-eight knot. A minimal alternating projection of Figure-eight knot is a projection of type I such that one of the tangles $\rm F$ and $\widehat{\rm F}$ contains two positive crossings and the other two negative crossings. This is a +achiral projection in the sense of Definition~2.6.

\item The condition ${\rm F}  \sim {\rm F}^*$ implies that flypes may be involved. This means that we do not have necessarily a minimal achiral projection. However, we shall give in \S5 a constructive proof for the existence of an achiral projection for the +AAA knots. More precisely, we describe a construction producing a non-alternating $*$-visible tangle $\mathcal F$ isotopic to ${\rm F}$.
\end{enumerate}
\end{remark}

\section{+AAA visibility Theorem 5.1}

In this section, we assume that {\bf projections and tangles are arborescent}.
Let us state +AAA Visibility Theorem 5.1 which deals with the existence of an achiral projection for +AAA knots.
\begin{theorem}[+ AAA Visibility Theorem]
\label{AAA}
Let $K \subset S^3$ be a prime, alternating, arborescent knot. Suppose further that $K$ is +achiral. Then there exists a projection $\Pi_K \subset S^2$ of $K$ (non necessarily minimal) and a diffeomeorphism $\Phi: S^2 \longrightarrow S^2$ of order 4 such that:
\begin{itemize}
\item[1.] $\Phi$ preserves the orientation of $S^2$.
\item[2.] $\Phi$ preserves the orientation of the projection.
\item[3.] $\Phi(\Pi_K) = \widehat {\Pi}_K$ where $\widehat{\Pi}_K$ denotes the image of $\Pi_K$ by reflection through the projection plane.
\end{itemize}
\end{theorem}
 In \S5.1, we define the depth of an arborescent tangle. The general form of a primary tangle $\rm F$ of a projection of type I is described in \S5.2.  By the analysis in \S5.2, we deduce that the obstacles to an achiral projection are essentially concentrated in the minimal central tangle of the primary tangle (Definition 5.10). Consequently, our proof of +AAA Visibility Theorem (outlined in \S5.4) is done essentially by induction on the depth of irreducible minimal central tangles; the $\alpha$-move (described in \S 5.3) provides the induction step while revealing the symmetry under $\rm R^*$.

\subsection{Essential Conway circles and Depth of tangles}

In this subsection, we introduce the notion of {\bf essential Conway circle} of a link projection $\Pi$ and we define the {\bf depth of a tangle} on which the proof by induction of +AAA Visibility Theorem 5.1 is based.

\begin{definition}
A {\bf rational tangle} is a tangle in which all the {\bf canonical Conway circles} (see Definition 8.10) are concentric and are such that the innermost circle bounds a disc containing only a spire (see the definition in the Appendix). A rational tangle of a link projection $\Pi$ is {\bf maximal} if it is not strictly contained in a larger rational tangle of $\Pi$.
\end{definition}

\begin{definition}
{\bf An essential Conway circle} of a projection is a canonical Conway circle which is not properly contained in a maximal rational tangle. 
\end{definition}

This means that for a maximal rational tangle of a link projection, only its boundary is essential. 


\begin{example}
Consider the tangle $\tau$ described in Fig. 9; the canonical Conway circles $\gamma_1, \dots ,\gamma_5$ are essential while $(\delta)$  is not an essential Conway circle. The circle $\gamma_0$ is not necessarily a canonical Conway circle of $\Pi$.
\end{example}

\begin{figure}[h!]
\includegraphics[scale=0.35]{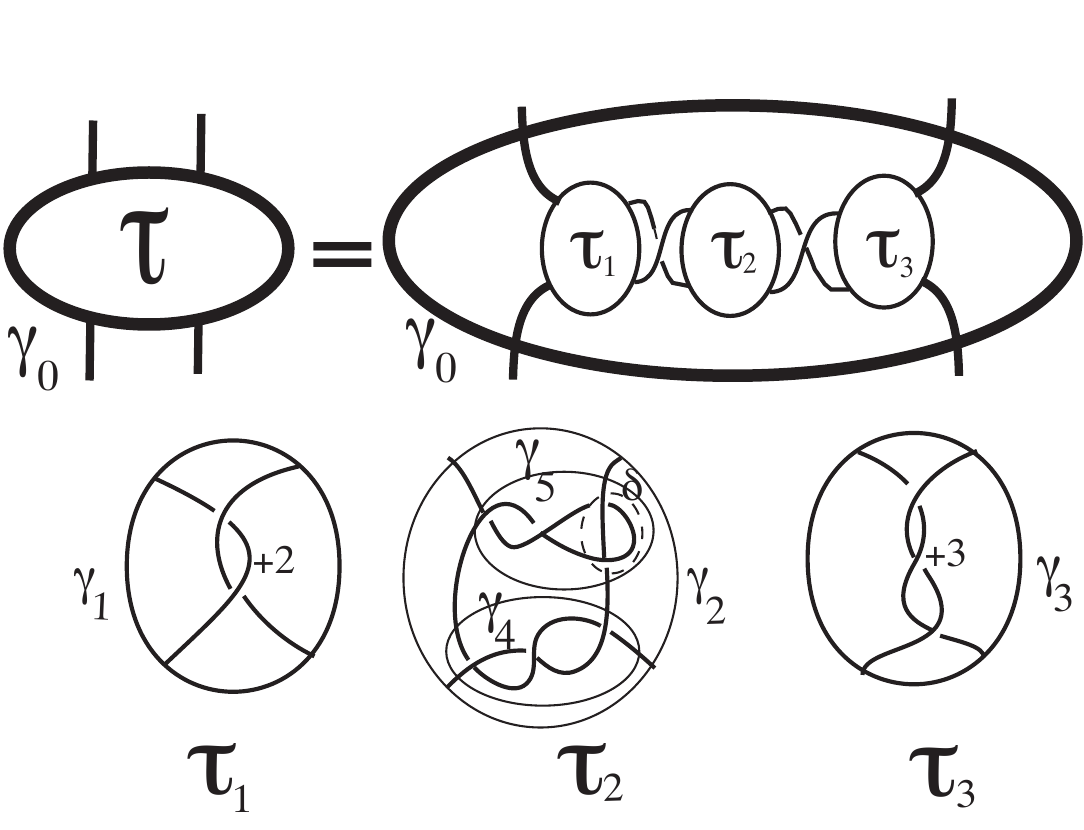}
\caption{A tangle with its essential Conway circles.}
\end{figure}

Let $\rm H$ be an (arborescent) tangle with its underlying disc $\Delta$.

Denote by $\mathcal{C}_{\Delta}^{ess}$ the set of essential Conway circles contained in $\Delta$. These circles divide $\Delta$ into connected planar surfaces. Let $\Sigma$ be such a connected planar surface which is not a disc. Denote its boundary by $\partial \Sigma  = \{\gamma, \gamma_1, \dots ,\gamma_n\}$ where inside $\Delta$, $\gamma$ bounds a disc $\Delta_{\Sigma}$ which contains all the other components  $\{ \gamma_1, \dots , \gamma_n\}$ of $\partial \Sigma$. Each $\gamma_i$ for $i=1, \dots, n$ bounds a disc in $\Delta$ which is called an {\bf inner disc} of the planar surface of $\Sigma$.

We denote by $\mathcal{P}_{\Delta}$ the set of planar surfaces determined by $\mathcal{C}_{\Delta}^{ess}$. We introduce an order relation on $\mathcal{P}_{\Delta}$ as follows.

\begin{definition}
Let $\Sigma^1$ and $\Sigma^2$ be two elements of $\mathcal{P}_{\Delta}$. We write $\Sigma^1 > \Sigma^2$ if $\Sigma^2$ is contained in an inner disc of $\Sigma^1$.
\end{definition}
Consider the  diagrams $(\Sigma_i , \Gamma_i)$ with $\Gamma_i = \Sigma_i \cap \Pi$. Since $\rm H$ is arborescent, each of these diagrams is either a maximal rational tangle or a twisted band diagram. Let us denote the set of these elements by $ \mathcal S_{\Delta}$. For $ \mathcal S_{\Delta}$, the order relation can be derived as follows.

\begin{definition}
 Let $(\Sigma_1 , \Gamma_1 )$ and $(\Sigma_2 , \Gamma_2 )$ be two elements of $\mathcal S_{\Delta}$ associated with $\Sigma_1$ and $\Sigma_2$. Define the order relation:
$$(\Sigma_1 , \Gamma_1 ) > (\Sigma_2 , \Gamma_2 )  \Longleftrightarrow \Sigma_1 >\Sigma_2.$$
\end{definition}

\begin{remark} 
A minimal element of $\mathcal{P}_{\Delta}$ is the underlying disc of a maximal rational tangle and a minimal element in $ \mathcal S_{\Delta}$ is a maximal rational tangle.
\end{remark} 

\begin{definition}
Let $\rm H$ be a tangle in an arborescent link projection. Let $l_\eta$ be the maximum of the lengths $l$ of descending chains $\Sigma_1 > \Sigma_2 > \cdots > \Sigma_{l+1}$ for the order relation $>$ on the set $\mathcal{P}_{\Delta}$. (Note that $l_\eta$ is the number of symbols $>$ in the longest descending chain).

The {\bf depth} $\mu (\rm H)$ of $\rm H$ is defined as: 
$$\begin{cases}
  \mu(\rm H)= l_\eta & \text{if }\partial \rm H \,\text{is an essential Conway circle of $\Pi$}. \\
   \mu(\rm H)= l_\eta +1& \text{if }\partial \rm H \,\text{is not an essential Conway circle of $\Pi$ and}\,\,   \rm H \,\\
& \text{is not a rational tangle}.
   \end{cases}$$
\end{definition}

\begin{figure}[h!]
\includegraphics[scale=0.4]{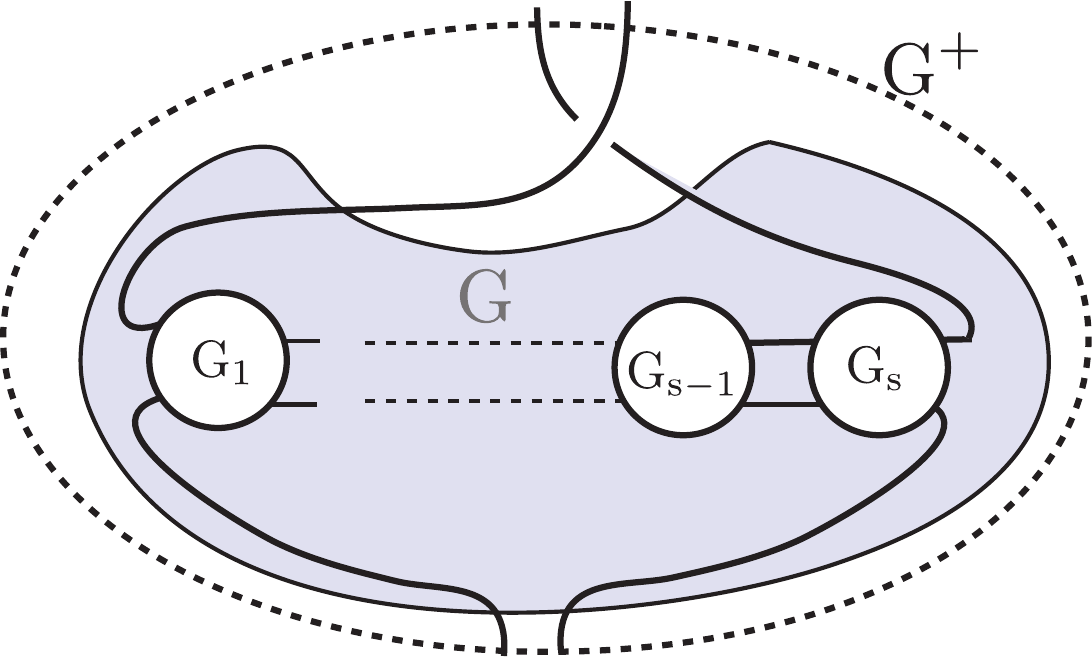}  
\caption{An irreducible tangle.}
\end{figure}

\begin{notation}
$^+\rm G$ is the tangle $\rm G$ enriched by an additional crossing north of $\rm G$ (Fig. 10).

$_+\rm G$ is the tangle $\rm G$ enriched by an additional crossing south of $\rm G$ (Fig. 16).
\end{notation}

\begin{definition} Let $\rm G$ be an arborescent sum tangle $ \rm G_1 \# \rm G_2 \cdots \#\rm G_s$ with $\rm s \geq 2$. The tangle $^+\rm G$ is {\bf irreducible} if the additional crossing to form $^+\rm G$ does not belong to the support band of the sum tangle $\rm G$ (as shown in Fig. 10).
\end{definition}

\begin{example}\ 
\begin{itemize}
\item[1)] Rational tangles are of depth $0$. 
\item[2)] Consider the tangle $^+\rm G$ of Fig. 20. Since $\mu(\rm G_1)= \mu(\rm G_2)=0$, we have $\mu(\rm G) = 1$ and $ \mu(^+\rm G)=2$. 
\end{itemize}
\end{example}

\subsection{Primary tangle of the +AAA knots and\\ its minimal central tangle}
\subsubsection{General form of a $*$-equivalent tangle}
As noted at the beginning of \S 5, in this section we only consider the knot projections and the tangles that are arborescent.
We first define the principal decomposition of an arborescent tangle.

\begin{figure}[h!]
\includegraphics[scale=0.60]{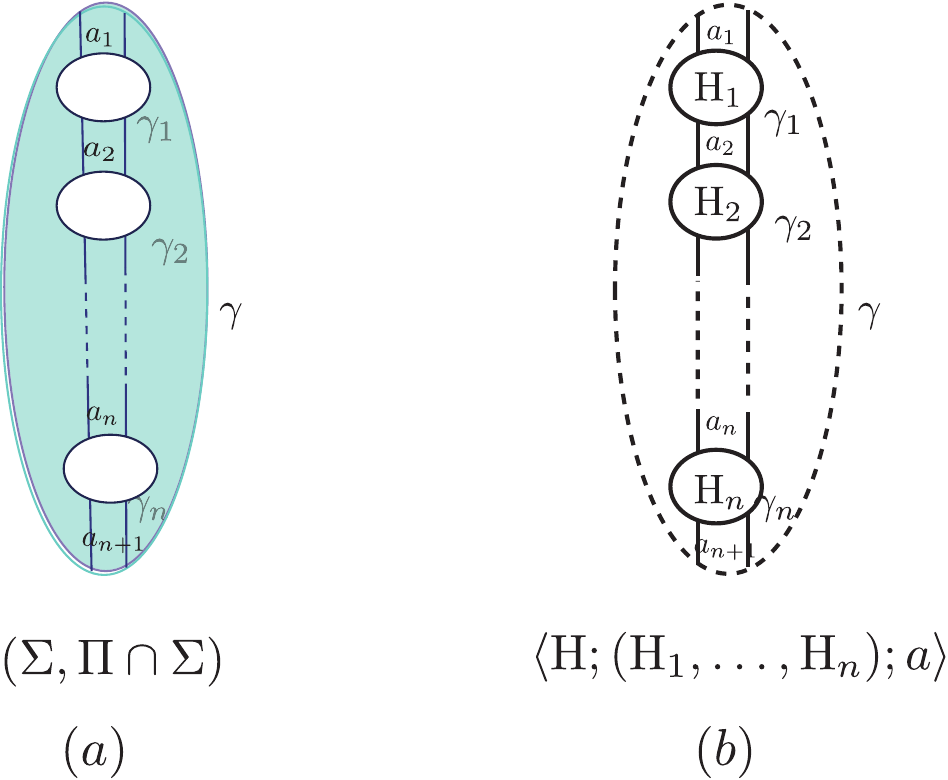}
\caption{$(a)$ Principal twisted band diagram $(\Sigma, \Gamma=\Pi \cap \Sigma)$ of $\rm H$; $(b)$ Principal decomposition of $\rm H$.}
\end{figure}

Let $\rm{H}=(\Delta, \Pi \cap \Delta)$ be a non-rational tangle. Denote its boundary by $\gamma$.

Consider the planar connected surface $\Sigma$ in $\Delta$ (Fig. 11) which has $\gamma$ in its boundary. Denote the other boundary components of $\Sigma$ by $\gamma_1, \dots, \gamma_n$; these are essential Conway circles. With $ \Gamma= \Pi \cap \Sigma$, the twisted band diagram $(\Sigma, \Gamma)$ is called the {\bf principal twisted band diagram} of $\rm H$ with weight $a =\sum a_i$. 
\begin{definition}
The principal twisted band diagram $(\Sigma, \Gamma)$ defines the {\bf principal decomposition} $\langle \rm H ;( \rm H_1, \dots, {\rm H}_n); a  \rangle$ of $\rm H$ in $n$ tangles $\rm H_i$ with $ \partial \rm H_i = \gamma_i$ for $i=1, \dots ,n$.

The {\bf breadth} $ l(\rm H)=n$ is the number of tangles which are the constituents of the principal decomposition of $\rm H$. (Fig. 11).
\end{definition}
Let $\rm G$ be an arborescent tangle. Using the arborescent structure of $\rm G$ and therefore of $^+\rm G$, we have the following definition of an irreducible tangle $^+\rm G$.
\begin{definition}
If the depth $ \mu(^+\rm G)$ satisfies $\mu(^+\rm G)= \mu(\rm G) +1$, the arborescent tangle $^+\rm G$ is {\bf irreducible}. So the principal twisted band diagram of $^+\rm G$ is a twisted annulus (\S 8.1), that is, $^+\rm G$ has the principal decomposition $\langle {^+\rm G}; {(\rm G)}; \pm1 \rangle$.
\end{definition} 
 
Let ${\rm H}$ be an arborescent $*$-equivalent tangle such that $ l{(\rm H)}=n \geq 2$ and its principal decomposition is $\langle {\rm H} ;({\rm H}_1, \dots, {\rm H}_n);a  \rangle$.

We have 4 possible cases for $\rm H$ according  the parity of $n$ and of $a$:

Without loss of generality, we can assume that $a=0$ or $a=1$.

If $ l({\rm H})=n=2k+1$ with $k \geq 1$, we define the {\bf central tangle} ${\cal K}(\rm H)$ of ${\rm H}$  as: 
\[ 
{\cal K}({\rm H})\ =\begin{cases}
  {\rm H}_{j+1} &{\rm if} \hspace{.15cm} l({\rm H})=n=2k+1 \hspace{.15cm} {\rm and}  \hspace{.15cm} a=0 .\\
  {^+\rm H}_{k+1}  &{\rm if} \hspace{.15cm}  l({\rm H})=n=2k+1 \hspace{.15cm} {\rm and}  \hspace{.15cm} a=1.
 \end{cases} 
 \]

The canonical decomposition of an alternating projection (see Appendix) and Key Theorem 3.1 imply the following properties:
\begin{figure}
\includegraphics[scale=0.5]{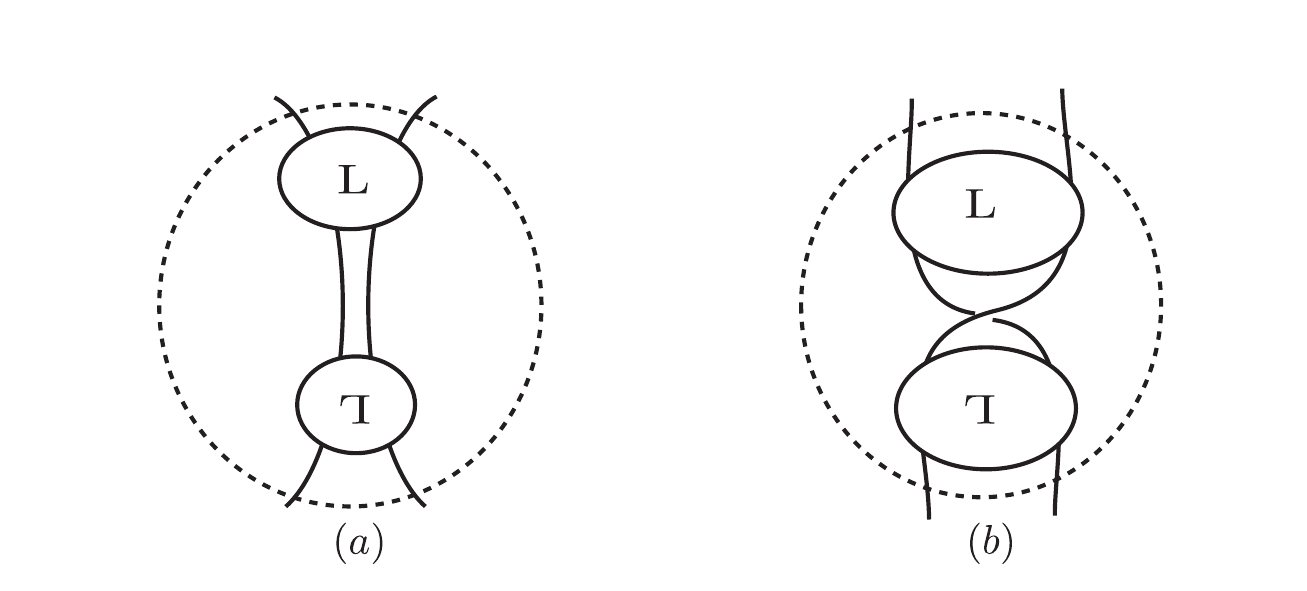}
\caption{Case with $l({\rm H})=n=2k$.}
\end{figure}

\begin{enumerate}
\item Suppose that {$n = 2k$ and $a=0$} (Fig. 12($a$)). The center of half-turn rotation $\rm R^*$ is between ${\rm H}_k$ and ${\rm H}_{k+1}$. Since  ${\rm H}$ is  $*$-equivalent, we have for $i = 1 , \dots , k$
$${\rm H}_{n+1-i} \sim {\rm H}_{i}^*.$$ 
Since the tangles ${\rm H}_{i}$ and ${\rm H}_{n+1-i}$ are disjoint, we can do flypes to get the {\bf pairing property} of ${\rm H}_i$, i.e.,
$${\rm H}_{n+1-i} \equiv {\rm H}_{i}^*.$$
and ${\rm H}_{n+1-i}$ is called the {\bf pairing} of ${\rm H}_{i}$.
By grouping the first $k$ tangles ${\rm H}_{i}$ in tangle $\rm L$, we get the pairing of $\rm L$ and therefore the $*$-visibility of ${\rm H}$.

 \item Suppose that {$n = 2k$ and $a=1$} (Fig. 12($b$)). 
The visible crossing can be moved by flypes to the center of the principal band and as in the previous case, we get the pairing for each $i = 1 , \dots , k$
 $${\rm H}_{n+1-i} \equiv {\rm H}_{i}^*$$ 
and therefore the $*$-visibility of $\rm H$.

 \begin{figure}[h!]
\includegraphics[scale=0.58]{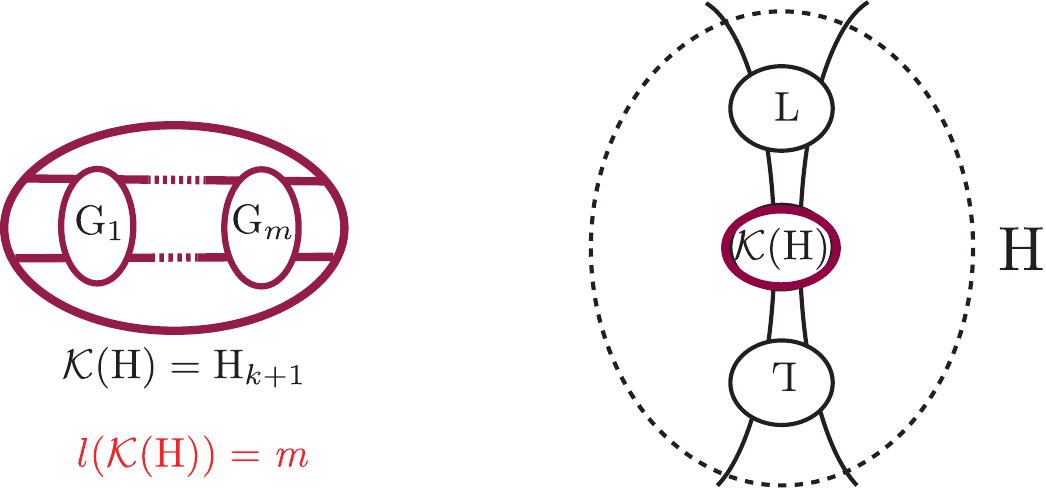}
\caption{Case with $l({\rm H})=n=2k+1$ and $a=0$.}
\end{figure}

 \begin{figure}[h!]
\includegraphics[scale=0.43]{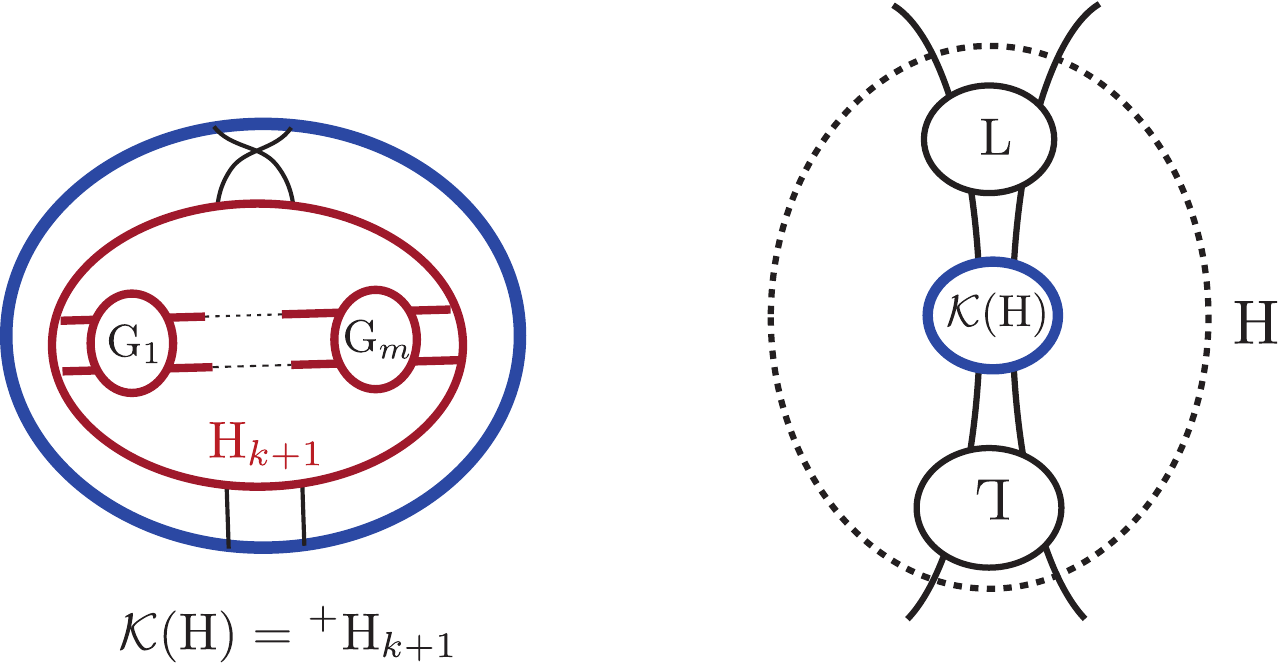}
\caption{Case with $l({\rm H})=n=2k+1$ and weight $a=1$.}
\end{figure}

$n=2k+1$ with $a=0$ or $a=1$. the pairing property is fulfilled for $i = 1 , \dots , k$:
$${\rm H}_{n+1-i} \equiv {\rm H}_{i}^*.$$
The center of $\rm R^*$ is the ``middle" of ${\cal K}(\rm H)$.
 
\item Suppose that {$n = 2k+1$ and $a=0$} (Fig. 13). 
Since $\rm H$ is $*$-equivalent, ${\cal K}(\rm H)$ is $*$-equivalent. Furthermore 
 ${\rm H}$ is slightly $*$-visible if ${\cal K}(\rm H)=\rm H_{k+1}$ is slightly $*$-visible.

\item Suppose that {$n = 2k+1$ and $a=1$} (Fig. 14).
Again with the flypes, the crossing can be moved to north of ${\rm H}_{k+1}$. In the case that ${\rm H}_{k+1}$ is not rational, the tangle  ${\cal K}(\rm H)= {^+{\rm H}_{k+1}}$ is irreducible.

Since $\rm H$ is $*$-equivalent, the pairing property implies that $^+\rm H_{k+1}$ is also $*$-equivalent. 
In addition, $\rm H$ is slightly $*$-visible if ${\cal K}(\rm H)=^+\rm H_{k+1}$ is slightly $*$-visible.
\end{enumerate}
\noindent {\bf Summary.} 
Let $\rm H$ be an arborescent alternating tangle with $ l(\rm H)=n \geq 2$. 
\begin{itemize}
\item[(a)] If $\rm H$ is $*$-equivalent, then $\rm H$ has the pairing property.
\item[(b)] $\rm H$ is $*$-equivalent if and only if $\rm H$ has the pairing property and ${\cal K}(\rm H)$ is $*$-equivalent. Furthermore,
\item[(c)] $\rm H$ is slightly $*$-visible if and only if $\rm H$ has the pairing property and ${\cal K}(\rm H)$ is slightly $*$-visible.
\end{itemize}

\subsubsection{Minimal central tangle of a primary tangle}
Let $\rm H$ be an alternating arborescent tangle with the principal decomposition:
$$\langle {\rm H}; ({\rm H_1},{\rm H_2}, \dots,{\rm H}_n); a \rangle.$$

We now extend the notion of central tangle to the set of tangles that are either of even breadth, or rational, or irreducible. If $\rm H$ is such a tangle, we define its central tangle ${\cal K(\rm H)}$ as a trivial horizontal or vertical tangle depending on whether the Seifert arcs of $\rm H$ are horizontal or vertical.

Let us denote the trivial (horizontal or vertical) tangle by ${\cal T_\emptyset}$.

\begin{definition}
The {\bf central tangle} ${\cal K}(\rm H)$ is:
\[ {\cal K}({\rm H})\ =\begin{cases}
            {\cal T_\emptyset} &
             \text{\rm{if $H$ is either rational or irreducible or}} 
             \hspace{.15cm} l({\rm H}) =n=2k.\\
           {\rm H}_{1} &{\rm if} \hspace{.15cm} l({\rm H})=1 \hspace{.15cm} {\rm and}  \hspace{.15cm}  a \equiv 0 (\rm mod \ 2) \hspace{.15cm} {\rm with}  \hspace{.15cm}  a \neq 0.\\
              {^+\rm H}_{1}  &{\rm if} \hspace{.15cm}  l({\rm H})=1 \hspace{.15cm} {\rm and}  \hspace{.15cm}  a \equiv 1 (\rm mod \ 2) \hspace{.15cm} {\rm with}  \hspace{.15cm}  a \neq 1. \\
              {\rm H}_{k+1}  &{\rm if} \hspace{.15cm} l({\rm H})=n=2k+1\,  (\geq 3)\hspace{.15cm} {\rm and}  \hspace{.15cm}  a \equiv 0 (\rm mod \ 2).\\
         {^+\rm H}_{k+1} &{\rm if} \hspace{.15cm}  l({\rm H})=n=2k+1\, (\geq 3) \hspace{.15cm} {\rm and}  \hspace{.15cm}  a \equiv 1 (\rm mod \ 2).\\
 \end{cases} 
 \]
 \end{definition}

\begin{remark} 
If ${\cal K}({\rm H})={^+\rm H}_{j+1}$ and  ${\cal K}({\rm H})$ is not a rational tangle, ${\cal K}({\rm H})$ is irreducible.
\end{remark}
We recursively define the central tangle ${\cal K}^{r}(\rm H)$ by
$${\cal K}^{r}(\rm H)= {\cal K} ({\cal K}^{r-1}(\rm H)).$$

Let $K$ be a $+$AAA knot, $\Pi$ a minimal projection of $K$ and $\rm F$ a primary tangle of $\Pi$.

Consider now the descending sequence (4) of central tangles of the primary tangle ${\rm F}$ which begins with $\rm F$, ends with
$ {\cal K}_{min} (\rm F) $:
\begin{equation} 
{\rm F} \supset {\cal K}(\rm F)\supset {\cal K}^{2}(\rm F)\supset \cdots \supset {\cal K}_{min} (\rm F)
\end{equation}

and which is such that each term ${\cal K}^{r-1}(\rm F)$  is followed by ${\cal K}^{r}(\rm F)$. The last term of (4) is the {\bf minimal central tangle} $ {\cal K}_{min} (\rm F) $.

\begin{definition}
The {\bf minimal central tangle}  ${\cal K}_{min}(\rm F)$ of $\rm F$ is defined by the condition that ${\cal K}_{min}(\rm F )\neq {\cal T_\emptyset}$ and ${\cal K}({\cal K}_{min}(\rm F)) ={\cal T_\emptyset}$. 
\end{definition}
Consequence of the definition: ${\cal K}_{min}(\rm F)$ is either rational, or irreducible, or of even breadth.

From the analysis in \S 5.2.1, we can deduce the following important properties of $\rm F$:
\begin{itemize}
\item[(a)] Since $\rm F$ is $*$-equivalent, $\rm F$ has the pairing property.
\item[(b)] All the central tangles of the descending sequence (4) and in particular ${\cal K}_{min} (\rm F)$, inherit the $*$-equivalence and the pairing property of $\rm F$. 
\item[(c)] $\rm F$ is slightly $*$-visible if and only ${\cal K}_{min} (\rm F)$ is slightly $*$-visible.
\item[(d)] $\rm F$ is a $*$-tangle if and only if ${\cal K}_{min} (\rm F)$ is a $*$-tangle.
\end{itemize}

Therefore, the minimal central tangle of a primary tangle plays an essential role in the proof of +AAA Theorem 5.1.

 \begin{figure}[h!]
\includegraphics[scale=0.6]{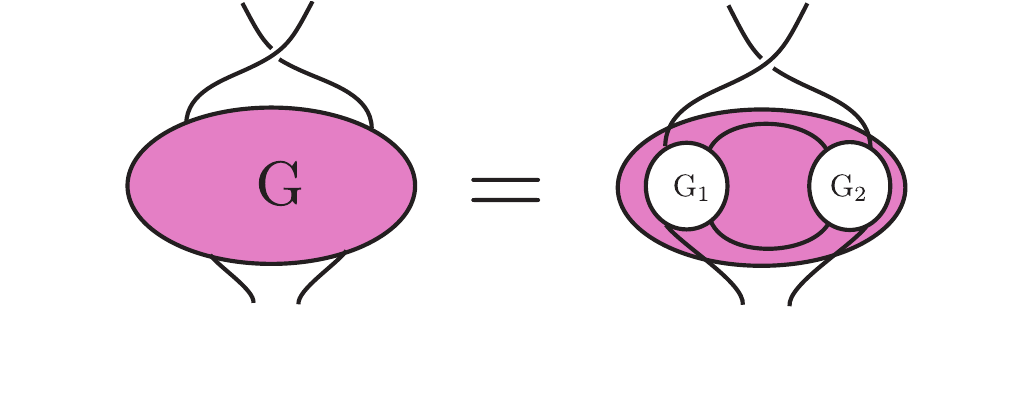}\vspace{-1em}
\caption{The tangle $\,^+{\rm G}$.}
\end{figure}
\subsection{Cross-plumbing and multiple cross-plumbing}
In this subsection, we introduce an $\alpha$-move on an irreducible tangle; this results in ``cross-plumbing" described as below. By its symmetry under $\rm R^*$, a cross-plumbing is a fundamental step towards the desired $*$-visibility stated in +AAA Theorem 5.1. If the irreducible tangle is $*$-equivalent, the tangles that make up the cross-plumbing are also $*$-equivalent (Proposition~5.1).

\subsubsection {$\alpha-$move and cross-plumbing}

Consider the irreducible tangle $^+{\rm G}$ where $\rm G$ is a sum tangle  $\rm G_1 \, \#  \, \rm G_2$ (Fig. 15); $\partial \rm G_1$ and $\partial \rm G_2$ are not necessarily essential Conway circles.
 \begin{figure}[h!]
\includegraphics[scale=0.5]{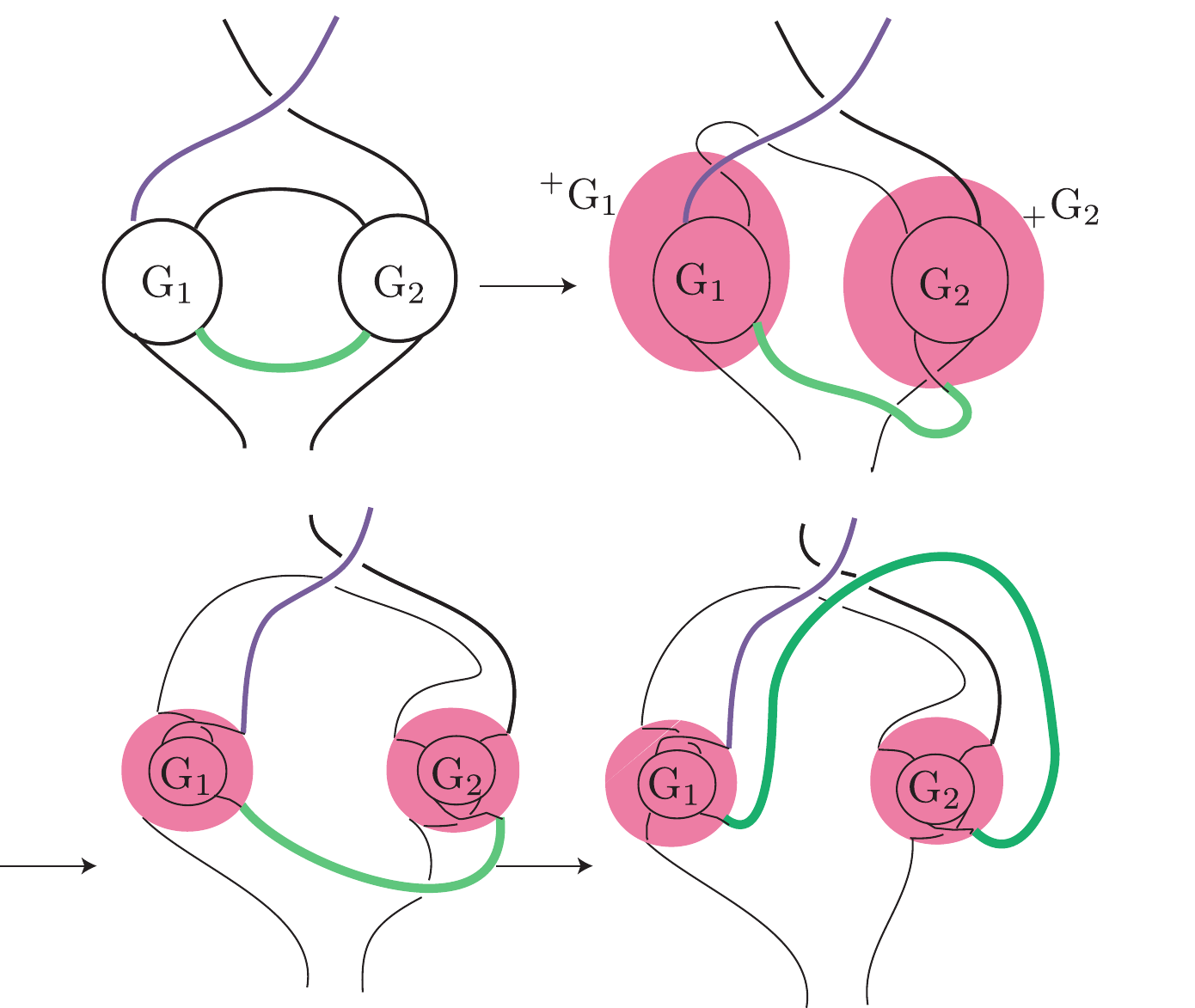}
\caption{An $\alpha-$move.}
\end{figure}

We now define an operation called {\bf $\alpha-$move} on the irreducible tangle $^+{\rm G}$ which consists 
\begin{itemize}
\item[(i)] first of creating four more crossings in performing two Reidemeister moves of type II, one north of $\rm G_1$ and the other south of $\rm G_2$, 
\item[(ii)] then by performing the isotopy of the strands as shown in Fig. 16. 
\end{itemize}

Finally, we get an untwisted main band that appears vertically on the figure and two bands plumbed on it, forming a $\bf X$-shape as shown in Fig. 17; one of the two band supports $^+\rm G_1$ and the other $_+\rm  G_ 2$. We say that $^+\rm G$ is a {\bf cross-plumbing} of $^+\rm G_1$ and $_+\rm  G_ 2$. 

\begin{figure}[h!]
\includegraphics[scale=0.45]{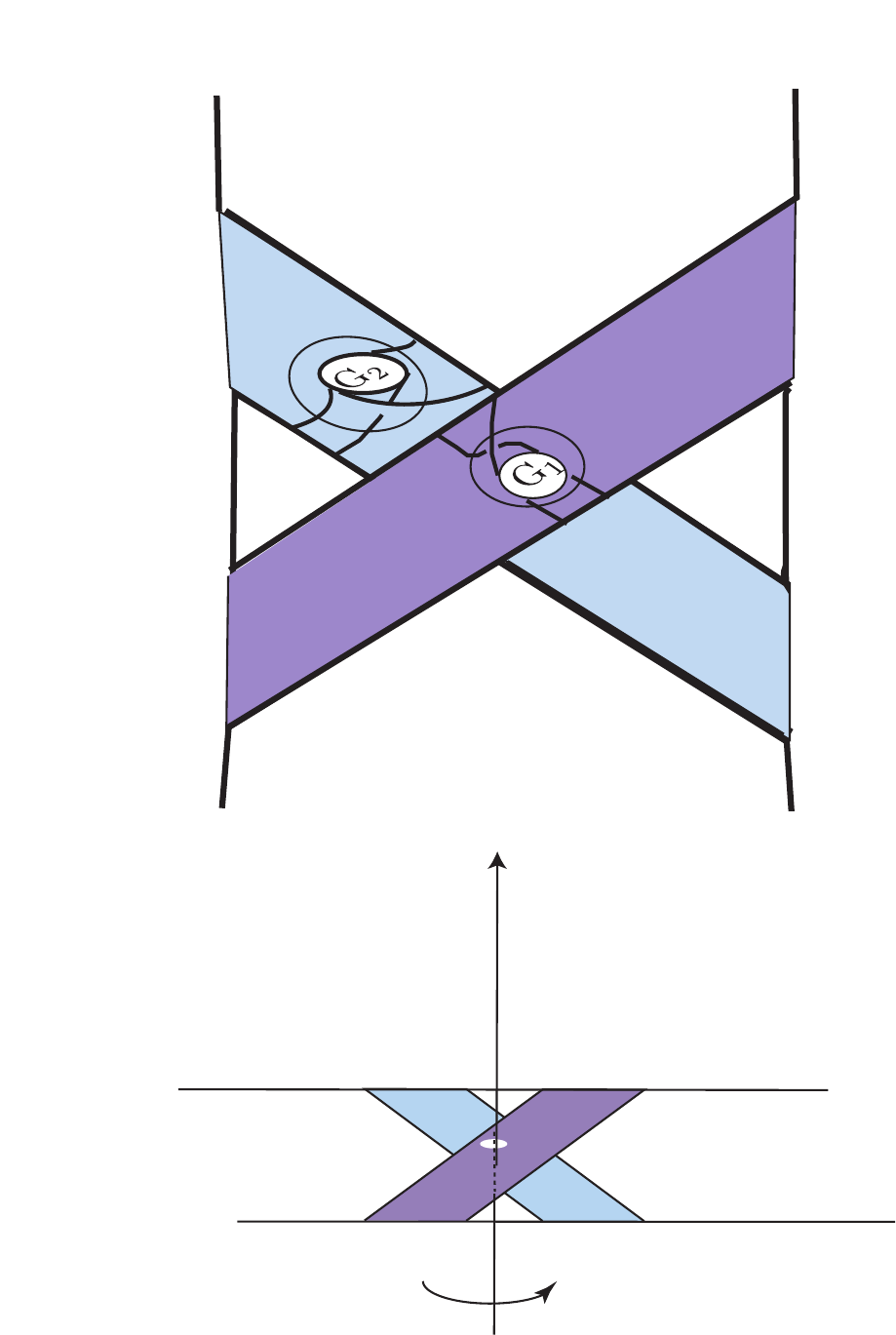}
\caption{A cross-plumbing.}
\end{figure}

\begin{figure}[h!]
\includegraphics[scale=0.352]{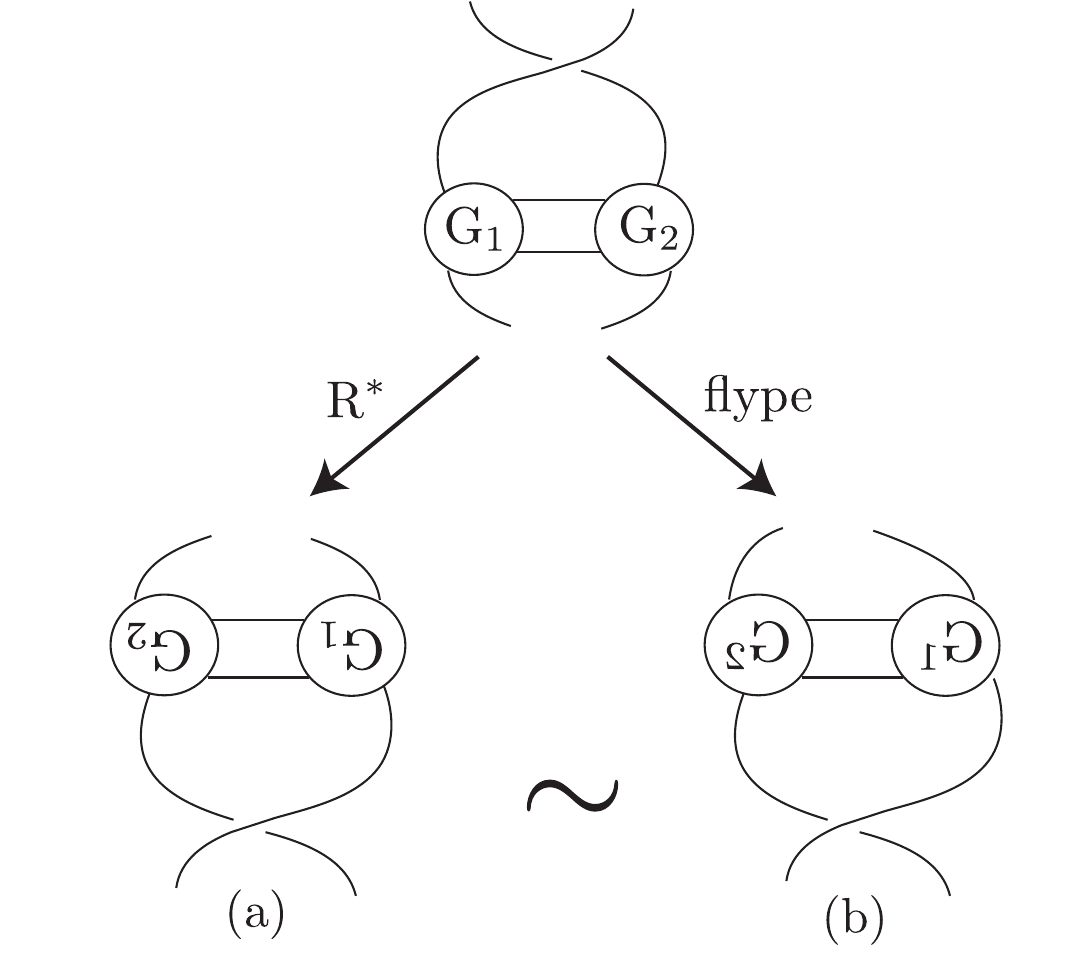}
\caption{$_+\rm G \sim _+\rm G^*$.}
\end{figure}

We have a simple but essential lemma:
\begin{lemma}
$\,^+\rm G$ is $*$-equivalent if and only if $\rm G$ is h-equivalent; analogously for  $\,_+\rm G$.
\end{lemma}

Let $\,^+\rm G$ be given as in Fig.15. We have:
\begin{lemma}
If $\,^+\rm G$ is $*$-equivalent, $\rm G_1$ and $\rm G_2$ are h-equivalent.
\end{lemma}
\begin{proof}
The proof is given by Fig. 18.
\end{proof}

From these two lemmas, one can deduce the following proposition which reduces to the case of $*$-equivalent tangles of smaller depth.
\begin{proposition}
If $^+\rm G$ is $*$-equivalent, $\,^+{\rm G}_i$ (and $\,_+{\rm G}_i$) is $*$-equivalent for $i=1,2.$
\end{proposition}
Proposition 5.1 can be generalised as follows:

\begin{figure}[h]
\includegraphics[scale=0.3]{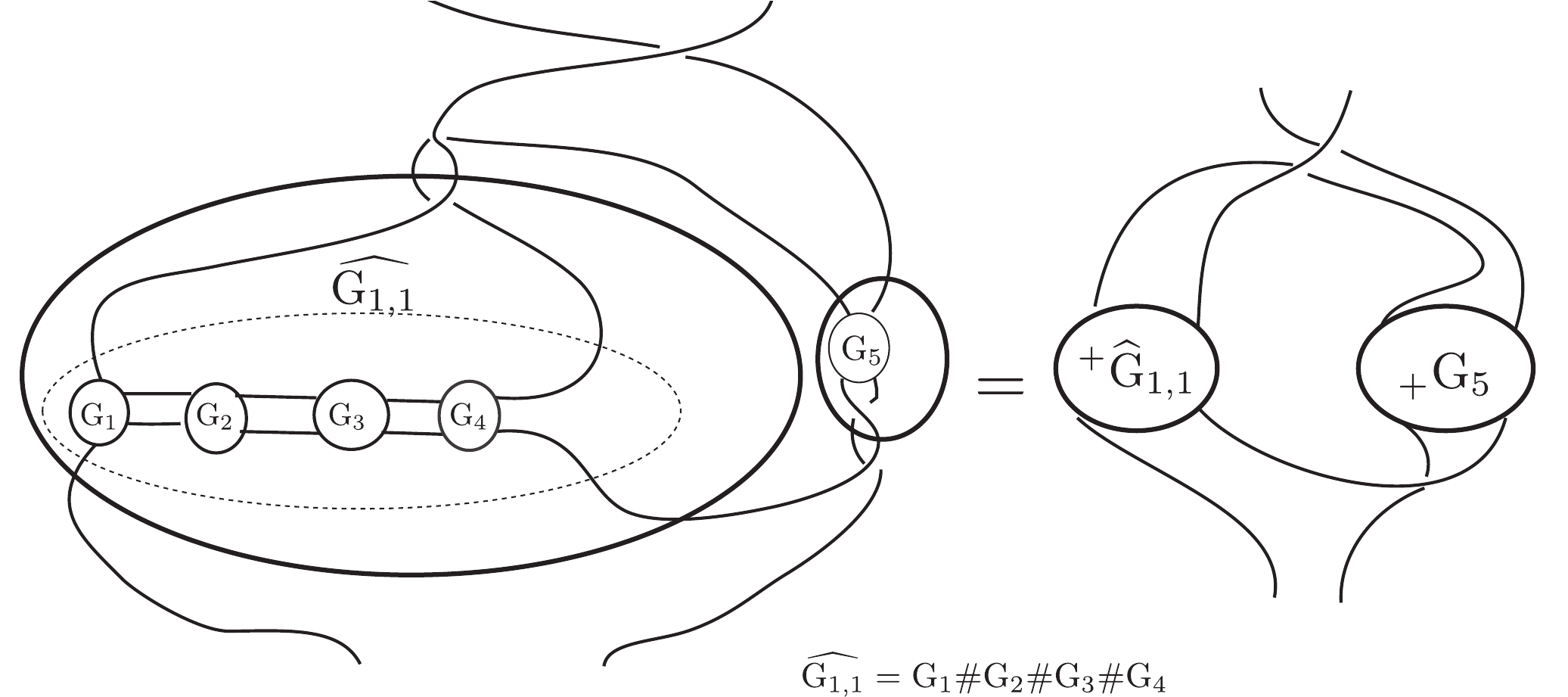}  
\caption{$^+G$ with a coarse sum on $\rm G={\rm G_1}\, \# \cdots \# \, {\rm G_5}$.}
\end{figure}
\begin{proposition}
Let $\rm G$ be a sum tangle with m ($\geq 2$) components, $\rm G={\rm G_1}\, \#.\cdots \#\,  {\rm G}_m$. Then
if $^+\rm G$ is $*$-equivalent, $^+{\rm G}_i$ is $*$-equivalent for each $i=1, \dots, m$ (idem for $_+{\rm G}_i$).
\end{proposition}
\subsubsection {Multiple cross-plumbing}
We generalise the notion of a cross-plumbing on an irreducible tangle $^+\rm G$ where $\rm G$ has two summands, that is $\rm G = \rm G_1 \, \# \rm  \, G_2$ to a multiple crossing-plumbing of  $^+\rm G$ where $\rm G$ has ${\rm s} (\geq 3)$ summands, that is $\rm G={\rm G_1}\#\cdots \# {\rm G_s}$.  

Consider the {\bf coarse sum} of $\rm G$ defined by the sum tangle $\rm G =\widehat{\rm G_{1,1}}\, \# \,{\rm G_s}$ where
$\widehat{\rm G_{1,1}}=\rm {G_1} \,\# \cdots  \# \,\rm{G_{s-1}}$ and 
$\rm G_s$ is the rightmost tangle of $\rm G=\linebreak \rm {G_1}\, \# \cdots \# \,  \rm G_s$.

We denote by $\widehat{\rm G_{1,j}}$ the sum tangle :
$$\widehat{\rm G_{1,j}}=\rm {G_1}\, \# \cdots  \#\, \rm G_{s-j}.$$ 

As shown in Fig.19, $\widehat{\rm G_{1,1}}$ and $\rm G_s$ now play the role of $\rm {G_1}$ and $\rm {G_2}$ of $^+\rm G$ described in Fig. 15. By a {\bf coarse $\alpha$-move}, we realise $^+\rm G$ as cross-plumbing of $^+\widehat{\rm G_{1,1}}$ with ${_+\rm G_s}$. In some sense, we ``extract " ${_+\rm G_s}$ from $^+\rm G$.

On $^+\widehat{\rm G_{1,1}}$, we again perform a coarse $\alpha$-move by considering $\widehat{\rm G_{1,1}}$ as coarse sum $\widehat{\rm G_{1,1}}=\widehat{\rm G_{1,2}} \# \rm G_{s-1}$. The tangle  $^+\widehat{\rm G_{1,1}}$ is then realised as {\bf coarse cross-plumbing} of $^+\widehat{\rm G_{1,2}}$ with $\rm_+{G_{s-1}}$.

And so on, we perform on each $^+\widehat{\rm G_{1,j}}$ for $\rm j=1, \dots, s-3$, a coarse $\alpha$-move which ``extracts " $\rm_+{G_{s-j}}$ until we get to $^+\widehat{\rm G_{1,s-2}}$.  At the end on $\widehat {\rm G_{1,s-2}}= G_1 \#  G_2$, we perform an $\alpha$-move which reveals this tangle as cross-plumbing of $^+{\rm G}_1$ with $_+{\rm G}_2$.

Therefore, by performing a sequence of $\alpha$-moves on  $^+\rm G$ , we realise $^+\rm G$ as a {\bf multiple cross-plumbing} of $^+{\rm G}_1$ with the $\rm {s-1}$ tangles  $_+{\rm G}_i$ where $i=2, \dots, \rm s$. 

\begin{remark}
By Proposition 5.2, if $^+\rm G$ is $*$-equivalent, $^+{\rm G}_i$ is $*$-equivalent for eachs $i=1, \dots, \rm s$ with $\rm s \geq 2$ (idem for $_+{\rm G}_i$). Therefore, if we can isotope every $^+{\rm G}_i$ into a $*$-visible tangle, a multiple cross-plumbing of these tangles gives rise to a $*$-visible tangle isotopic to $^+\rm G$.
\end{remark}

\subsection{Outline of the proof of Theorem 5.1}
Let $K$  be a +AAA knot. By Theorem 4.1, $K$ has a minimal projection $\Pi$ of type I (Fig. 7) with an invariant Haseman circle which decomposes $\Pi$ into two tangles $\rm F$ and $\widehat{\rm F}$ fulfilling the condition of being $*$-equivalent. Being $*$-equivalent tangles, $\rm F$ and $\widehat{\rm F}$ are both strictly or both slightly $*$-visible.

The goal is to show that $\rm F$ can be isotoped into a $*$-visible tangle $\cal F$. If so, $\widehat{\rm F}$ can also be isotoped into the $*$-visible tangle $\widehat{\cal F}$ and we are done: the resulting projection $ \Pi'$ isotopic to $\Pi$ is obtained from the ``plumbing" of $\cal F$ with $\widehat{\cal F}$ which is invariant under a twisted rotation of order 4 (as shown in Fig. 7).

Moreover, since the strict or slight $*$-visibility of $\rm F$ is completely determined by ${\cal K}_{min}(\rm F)$ (\S 5.2.2), it suffices to focus on $*$-equivalent minimal central tangles.

If ${\cal K}_{min}(\rm F)$ is a $*$-tangle, by using the pairing property on each term of the sequence (4), we can get a $*$-visible tangle isotopic to $\rm F$ and we are done. Therefore, our constructive proof of +AAA Visibility Theorem should only deal with $*$-equivalent minimal central tangles.

By \S 5.2.2, the $*$-equivalent minimal central tangles are either (1) of even breadth or (2) rational or (3) irreducible. 

(1) If ${\cal K}_{min}(\rm F)$ has its even breadth, ${\cal K}_{min}(\rm F)$ is strictly$*$-visible. By using the pairing property on each term of the sequence (4), we realise a $*$-visible primary tangle $\rm F_0$ which is alternating. We are done with $\cal F = \rm F_0$; as expected, the ``plumbing"  of $\rm F_0$ with $\widehat{\rm F_0}$ shown in Fig. 7 results in a minimal achiral projection via Definition 2.5.

(2) If ${\cal K}_{min}(\rm F)$ is rational, ${\cal K}_{min}(\rm F)$ is a $*$-tangle (in an appropriated plumbing form (see \S 5.6) or in its pillow-form (\cite{crom} Theorem 8.2), a rational tangle has many visible symmetries; by rotating it by an angle $\pi$ aroundt any principal axis (North-South, East-West) or any axis orthogonal to the projection plane produces the same isotopy of tangle). 

By using the pairing property on each term of the sequence(4), we can thus present a $*$-visible tangle isotopic to $\rm F$ which is not alternating in general. Therefore, the resulting achiral projection is not minimal in general.

(3) Our proof of +AAA Theorem 5.1 is  focused on $*$-equivalent minimal central tangles and is done by induction on their depth such that the induction hypothesis  $ ({\bf{{\rm P}_ n}})$ of +AAA Visibility Theorem 5.1 is
\begin{center}
 $ ({\bf{{\rm P}_ n}})$ \quad \quad
{\bf $*$-equivalent alternating arborescent tangles of depth $ \leq n$ are $*$-tangles.}
\end{center}

By (1) and (2), the only case remaining to be treated is the one where the $*$-equivalent minimal central tangles are irreducible.

 As we will see in \S5.5, to prove that an irreducible $*$-equivalent tangle is a $*$-tangle, we will perform multiple cross-plumbing which involves $*$-equivalent tangles of smaller depth and therefore, gives the induction step.

\subsection{Case of irreducible $*$-equivalent tangles}
The aim is to prove that the irreducible $*$-equivalent minimal central tangles ${\cal K}_{min}(\rm F)$ are $*$-tangles.

From now on, we note an irreducible $*$-equivalent tangle ${\cal K}_{min}(\rm F)$ by $^+\rm G$ of breadth $l(\rm G)=s \geq 2$.

\begin{figure}[h]
\includegraphics[scale=0.350]{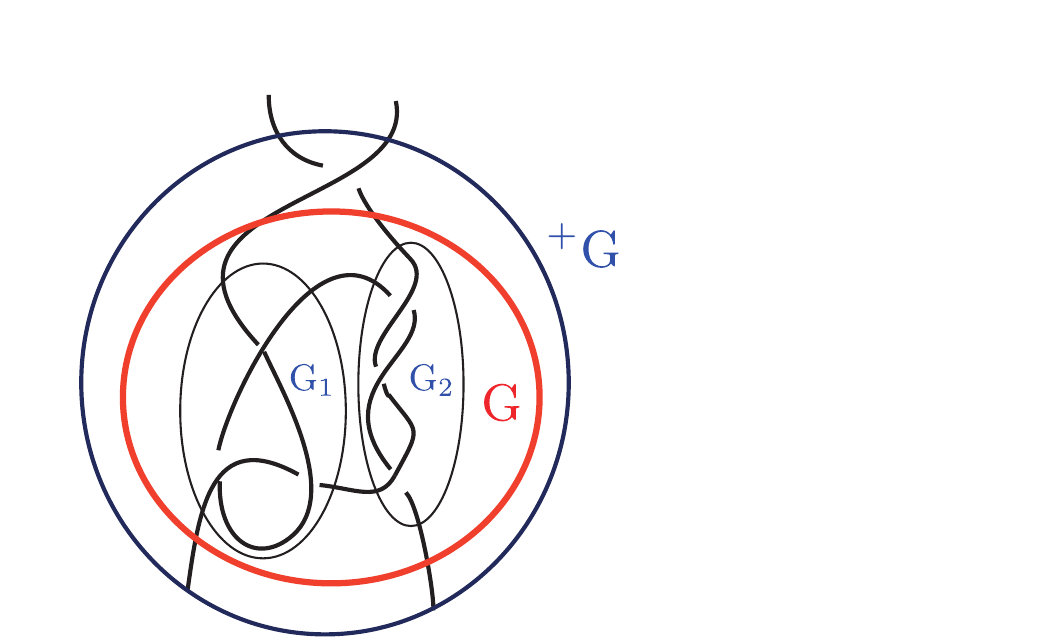}
\caption{A tangle $^+\rm G$ with  $ \mu(^+\rm G)=2$.}
\end{figure}
\noindent {\bf Claim.}
Let $^+\rm G$ be an irreducible tangle, i.e.  $^+\rm G$ has principal decomposition
$$\langle  ^+\rm G; (\rm G): \pm1\rangle.$$
Suppose that $\rm G$ has a principal decomposition
 $$\langle \rm G; (\rm G_1, \dots \rm G_s),0 \rangle.$$
We write $\mu_i=\mu (\rm G_i )$ for each $i=1,\dots, s$ and $\mu_{\rm M}=\max \{\mu (\rm G_1), \dots ,\mu (\rm G_s )\}$.

Since $\partial \rm G_i$ is an essential Conway circle of $\Pi$, we have
$ \mu (\rm G)= \mu_{\rm M} +1$ and
$$ \mu( ^+\rm G)= \mu( \rm G)+1=\mu_{\rm M}+2. $$

\noindent {\bf Consequence.} The minimal depth of an irreducible tangle is 2.
\vspace{.4cm}

$\bullet$  $\bf (\rm P_n)$ is true for $n=2$:

$(*)$ Case where $\mu (^+\rm G)=2$ and $\rm G$ is of breadth 2. 

The tangle $\rm G$ then has a principal decomposition:
$$\langle {\rm G};({ \rm G_1}, {\rm G_2});0 \rangle$$
where $\rm G_1$ and $\rm G_2$ are both rational. By performing an $\alpha$-move on $^+\rm G$, we get a cross-plumbing of $\rm ^+G_1$ and $\rm _+G_2$ which are both rational and therefore all two of the $*$-tangles. Via the cross-plumbing, we obtain an isotopic tangle to $^+\rm G$ which is invariant under $\rm R^*$.
 
$(**)$ Case where $\mu (^+\rm G)=2$ and $\rm G$ has breadth $l(\rm G)= \rm s \geq 3$. 

Since $\mu (^+\rm G)=2$, each $\rm G_j$ of the principal decomposition 
$$ \langle \rm G;( \rm G_1, \dots ,\rm G_s);0\rangle.$$
is rational for $\rm j=1, \dots, s$. Thus, just like $\rm G_j$, the tangles $_+\rm G _j$ and $^+\rm G _j$ being equally rational, are $*$-tangles. Via the multiple cross-plumbing described in \S 5.3.2, we can produce a $*$-visible tangle isotopic to $^+\rm G$.

Note that $\mu (^+\rm G)=2$ corresponds to the case where $\rm G$ is a Montesinos tangle (Definition 6.3).

$\bf (\rm P_2)$ being true, we can proceed to the induction step.

\begin{figure}[h!]
\includegraphics[scale=0.4]{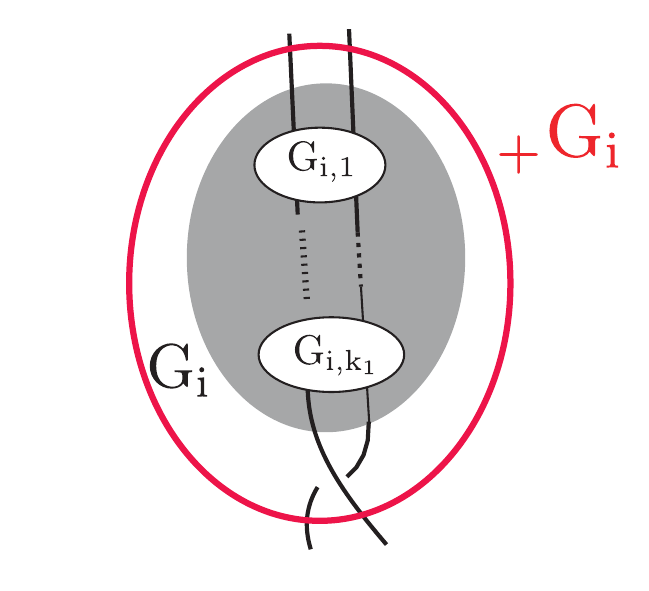}  
\caption{ $\mu (_+\rm G_i )= \mu (\rm G_i)$.}
\end{figure}

How are $\mu(_+ \rm G_i)$ and $ \mu( ^+\rm G_i)$ related to $\mu_i$?

\medskip
\noindent {\bf Claim.} For $i \neq 0$,
$$\mu(_+ \rm G_i) = \mu( \rm G_i)= \mu_i.$$
\begin{proof}
By definition of the principal decomposition of $\rm G$, the boundary $\partial \rm G_i$ for $i \neq 0$ is an essential Conway circle of $\Pi$ and the additional crossing south of $\rm G_i$ can be assimilated to the principal twisted band of $\rm G_i$ (see Fig. 21). So $_+ \rm G_i$ is not irreducible and $$\mu(_+ \rm G_i) = \mu_i.\vspace{-2em}$$
\end{proof}

$\bullet$ By induction hypothesis, $(\bf \rm P_{\it n})$ is true: each $*$-equivalent arborescent alternating tangle with depth $ \leq n$ is a $*$-tangle.

Now consider the case where $ ^+\rm G$ is a $*$-equivalent irreducible arborescent alternating tangle with $\mu {(^+\rm G)}= n+1$.

Let $^+\rm G$ be with $\mu {(^+\rm G)}= n+1$. 

As $\mu (^+\rm G)= 2 + \mu_{\rm M}$ where $\mu_{\rm M}=\max \{\mu (\rm G_1), \dots ,\mu (\rm G_s )\}$, we have 
$$ n+1= 2 + u_{\rm M}.$$
Hence $\mu_{\rm M}=n-1$. This implies that each $\mu_i \leq n-1$

Let $\rm G$ be with principal decomposition:
$$ \langle{ \rm G;( \rm G_1, \dots ,\rm G_s)};b \rangle.$$ 

By gathering the $b$ crossings into a tangle $\rm G_0$, we can consider $\rm G$ as a sum tangle with tangles of depth $\leq n-1$
$$\rm G= \rm G_0 \,\, \# \rm G_1 \, \#\cdots \, \# \,  \rm G_s.$$

Being rational, $^+\rm G_0$ is a $*$-tangle.

Proposition 5.2 implies that  $_+ \rm G_i$ and $^+ \rm G_i$ are $*$-equivalent for each $i=1, \dots, \rm s$.

As described in \S5.3.2, we consider a sequence of coarse sums and $\alpha$-moves which results in a multiple cross-plumbing on ${^+\rm G}$ with $*$-equivalent tangles.

Since $_+\rm G_i $ of depth $\mu_i \leq n-1$ is $*$-equivalent, by $(\bf \rm P_{\it n})$, $_+\rm G_i $ is also a $*$-tangle. Therefore, by multiple cross-plumbing of $*$-tangles, we can exhibit $*$-visibility on a non-alternating tangle $\cal{G}$ isotopic to $^+\rm G$.

So the statement $(\bf \rm P_{{\it n}+1})$ is true for all $n \geq1$.. 

Therefore, we have shown that $*$-equivalent irreducible tangles are $*$-tangles.

\medskip
\noindent {\bf Summary.} 
It follows that whenever ${\cal K}_{min}(\rm F)$ is rational or irreducible or of even breadth and endowed with the pairing property, we can isotope the primary tangle $\rm F$ in a $*$-visible tangle $\cal F$ (which is not necessarily alternating) , giving rise to an achiral projection for the +AAA knot $K$ via the projection of Fig. 7. See examples in \S 6.

\subsection{Rational tangles}
Let $T$ be a rational tangle with its underlying disc $\Delta$. Consider the canonical Conway circles which are concentric, delimiting twisted annuli except the most interior twisted band diagram which is a spire. We denote these twisted band diagrams by $P(x_i)$ with $i = 1 , \dots , u$; $P(x_u)$ is a spire and the other ones are twisted annuli.

As in Definition 5.3, we define an order relation between these twisted annuli and the spire:

$P(x_k) > P(x_j)$ if and only if $P(x_j)$ is contained in the inner disc of $P(x_k)$.

Let us consider the maximal chain $P(x_1) > \cdots  > P(x_u)$ of the rational tangle $T$. We denote this chain by $P(x_1 , \dots , x_u)$.  Odd weights prevent to exhibit a symmetric form. We could use $\alpha-$moves. However there is a better way to proceed. This argument is present more or less explicitly several times in the literature (see (\cite{kala}, {\cite{ haka}).

Since the projection is alternating, the signs of the weights alternate. Without any real loss of generality, we assume that $x_1 > 0$. Let $a_i = (-1)^{i+1}x_i$. As usual, we define the rational number $p/q$  by the continued fraction: 

$${p \over q} = {[a_1, a_2, \dots, a_n]}=a_1+ \dfrac{1}{a_2 +\dfrac{1}{{\ddots} +\dfrac{1}{a_n}}}$$
 \begin{figure}[h!]
\includegraphics[scale=0.4]{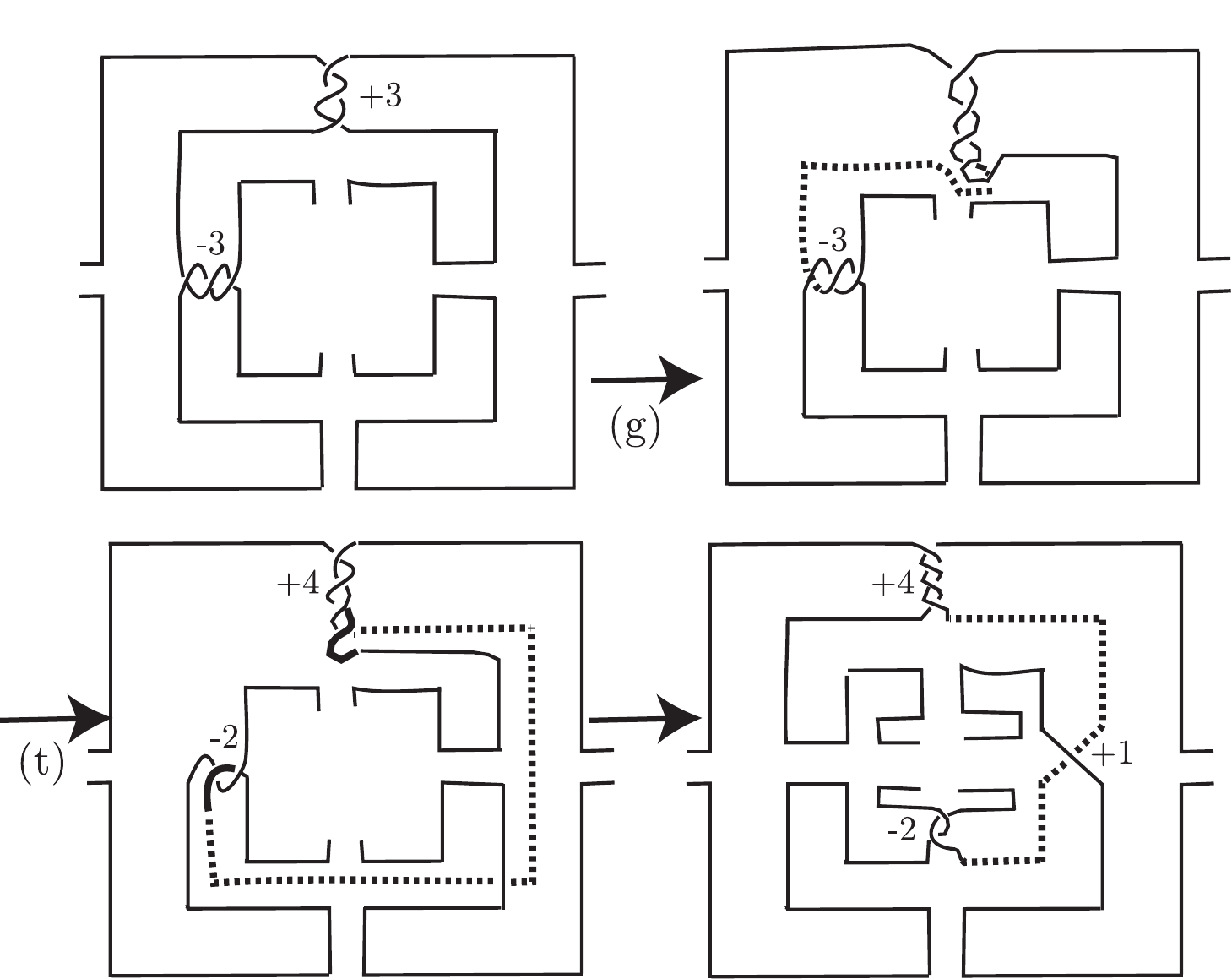} 
\caption{A gimmick (g) followed by a transfer move (t).}
\end{figure}

We denote this continued fraction expansion by $C(a_1 , \dots , a_u)$. Let $p/q$ be such that $p$ and $q$ are coprime and both odd. Consider $C(b_1 , \dots , b_v)$ the continued fraction expansion of $p/q$ with each $b_j$ even except for $b_v$ if $p$. Let $y_j = (-1)^{j+1}b_j$. It is possible to modify the plumbing $P(x_1 , \dots , x_u)$ in plumbing $P(y_1 , \dots , y_v)$ by a sequence  of operations which correspond to a $\pm$ blow-up in plumbing calculus (see Walter Neumann \cite{neum}) and also to Lagrange Formula (see Cromwell's book p.204 \cite{crom}) in continued fraction expansions. The operation modifies the plumbing 
$$P(z_1 , \dots , z_i , z_{i+1} , \dots , z_w)$$
 to the plumbing
 $$P(z_1 , \dots , z_{i-1} , z_i\pm1 , \pm1 , z_{i+1}\pm1 , z_{i+2} , \dots , z_w)$$

The corresponding tangles are isotopic (see Definition 4.1). Note that the plumbing notation takes care the signs elegantly. The  {\bf $\pm$ blow-up} operation is defined as follows. It is the combination of a {\bf gimmick} and of a {\bf transfer move} in the sense of Kauffman-Lambropoulou (\cite{kala}) including a rotation of angle $\pi /2$ for the inner tangle. The gimmick introduces two crossings of opposite sign by a Reidemeister move of type II at the end of a twist. The transfer move pushes a non-alternating arc. Its new position creates the $\pm 1$ between the $i$th and the $(i+1)$th entry in the plumbing. See Fig. 23.

At some point in the sequence of operations,  it is possible to encounter a weight equal to zero. In this case the following plumbings are isotopic.
\begin{align*}
P(z_1 , \dots , z_{i-1} , 0 , z_{i+1} , \dots , z_w) &\equiv P(z_1 , \dots , z_{i-1} + z_{i+1} , \dots , z_w)\\
P(z_1 , \dots , z_{w-2} , z_{w-1} , 0) &\equiv P(z_1 , \dots , z_{w-2})
\end{align*}

The isotopies are easily visible on the corresponding diagrams. Of course, this can also be checked on continued fractions. These last equivalences are called 0-absorptions by Walter Neumann.

\begin{remark}\ 
\begin{enumerate}
\item We assume that the first weight $x_1$ is $>0$, but there is no real difference in the arguments if we have $x_1 < 0$. Just use $-$blow- ups instead.
\item The last weight $y_v$ is odd if $p$ and $q$ are both odd. This fact does not prevent $P(y_1 , \dots , y_v)$ from being invariant by a half turn since the innermost circle contains a spire.
\end{enumerate}
\end{remark}

\begin{example}
Consider the rational tangle  $P(1 , -4)$. It is isotopic to the tangle $P(2 ,2 ,2 ,2)$ which exhibits $*$-visibility.
\end{example}
\begin{figure}[h!]
\includegraphics[scale=0.30]{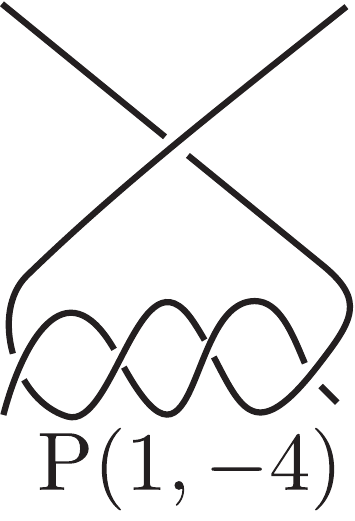}
\caption{The rational tangle $P(1, -4)$.}
\end{figure}

\section{The knots of Dasbach-Hougardy and Stoimenow}

Knots (not necessarily alternating) that are +achiral but not $-$achiral are rather rare among achiral knots. Among the 20 achiral knots (all alternating) with crossing number $c <12$, none is only +achiral. 

According to Hoste-Thistlethwaite-Weeks (\cite{hothwe}) we have:
\begin{itemize}
\item[1)] For $c = 12$ there are 54 alternating achiral knots. Exactly one of them is only +achiral. It was recognized by Haseman (it is her knot 59=60) and also earlier by Tait, with a vocabulary different from what is used  today. See \cite{quwe}. There are also 4 non-alternating achiral knots. No one is only +achiral.
\item[2)] For $c = 14$ there are 223 alternating  achiral knots. Among them 5 are only +achiral. There are also 51 non-alternating achiral knots. Exactly one of them is only +achiral.
\item[3)] For $c = 16$ there are 1049 alternating  achiral knots. Among them 40 are only +achiral. There are also 490 non-alternating achiral knots, with  25 only +achiral. 
\item[4)] All in all there are 1'701'935 non-trivial knots with $c \leq 16$. There are 491'327 alternating knots and 1'201'608 non-alternating ones. There are 1'892 achiral knots (including a surprising one with $15$ crossings). Among them 1'346 are alternating and 546 non-alternating. There are 82 knots which are only +achiral; 56 are alternating and 26 non-alternating.
\item[5)] There are therefore 1'290 alternating achiral knots with $c \leq 16$ which are $-$achiral. For all, the $-$achirality is visible on a minimal projection, according to our result \cite{erquwe}.
\end{itemize}
As alternating knots which are $-$achiral always have minimal achiral projections, we now only consider alternating knots that are only $+$achiral. In particular those for which there exists no achiral minimal projection. The first example of this class was given by Dasbach-Hougardy in \cite{daho}.  Since it is arborescent, according to \cite{erquwe}, it has a minimal projection of Type I as shown in Fig. 6. Its tangle $\rm F$ called {\bf Dasbach-Hougardy tangle} has the form shown in Fig. 24.

\begin{figure}
\includegraphics[scale=.5]{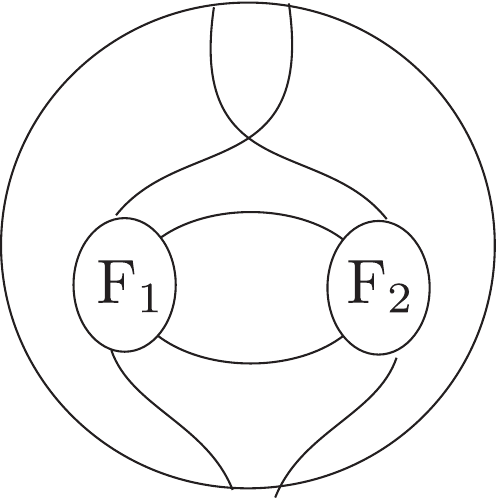} 
\caption{The shape of Dasbach-Hougardy tangle.}
\end{figure}

\begin{definition}
Let $K$ be an arborescent alternating knot such that $K$ has an alternating projection of Type I with its primary tangle $\rm F$ formed of two tangles $\rm F_1$ and $\rm F_2$ as described in Fig. 24. If $K$ satisfies the following conditions:
\begin{enumerate}
\item $\rm F \sim \rm F^*$ 
\item $\rm F_1$ is not flype-equivalent to one of the tangles  $\rm F_2$, $\rm F_2 ^v$, $\rm F_2 ^h$ and $\rm F_2 ^3$.
\end{enumerate}
then $K$ is called a {\bf DH-knot}.
\end{definition}

\begin{proposition}\hspace{-1ex}
Let $K$ be a DH-knot. Then $K$ is +achiral but not $-$achiral.
\end{proposition}
\begin{proof}
Condition 1) implies that $K$ is +achiral and
Condition 2) implies that $\rm F$ is not flype-equivalent to $\rm F^v$ or $\rm F^h$. Therefore, by Proposition 6.3 in \cite{erquwe}, $K$ is not $-$achiral.
\end{proof}

\begin{remark}\ 
\begin{enumerate}
\item   The north crossing on the irreducible tangle in Fig. 25 prevents the minimal  projection from being achiral. In other words, such a knot has no minimal projection which is +achiral.  However, the $\alpha$-move provides a non-minimal achiral projection. 
\item If the tangles $\rm F_i$ are not arborescent. there is not always a minimal projection where the +achirality is visible. But we do not always have a method to present an achiral projection.

Denote by $c_i$ the number of crossings of $F_i$. Hence the number $c$ of crossings of the minimal projection of the knot is $2 (c_1 + c_2 + 1)$. To avoid degeneracy we must have $c_i \geq 3$. 

Let us restrict $F_i$ to rational tangles. 
Let $C(a_1 , \dots , a_u)$ be the Conway word for a rational tangle.

There are two rational tangles with crossing number equal to 3: $C(1 ,2)$ and $C(3)$. Hence there is exactly one knot with $c = 14$. It is the original Dasbach-Hougardy knot with $\rm F_1 = C(1 , 2)$ and $\rm F_2 = C(3)$. Its Hoste-Thistlethwaite-Weeks notation is $14-10435 (a)$. 
\end{enumerate}
\end{remark}

We now consider DH-knots with 16 crossings. There are 4 rational tangles with crossing number equal to 4: $C(1 , 1 , 2)$ , $C(1 , 3)$ , $C(2 , 2)$ and $C(4)$. We now consider DH-knots with 16 crossings. There are 4 rational tangles with crossing numbers equal to 4: $C(1 , 1 , 2)$ , $C(1 , 3)$ , $C(2 , 2)$ and $C(4)$. But $C(1 , 3)$ and $C(4)$ should be discarded because we would  get a link. Therefore, we can construct  DH-knots as follows.

\begin{itemize}
\item[1)] $\rm F_1 = C(1 ,2)$ and $\rm F_2 = C(1 , 1 , 2)$; this is the knot $16-178893 (a)$.
\item[2)] $\rm F_1 = C(1 ,2)$ and $\rm F_2 = C(2 , 2)$; this is the knot $16-125918 (a)$.
\item[3)] $\rm F_1 = C(3)$ and $\rm F_2 = C(1 , 1 , 2)$; this is the knot $16-223267 (a)$.
\item[4)] $\rm F_1 = C(3)$ and $\rm F_2 = C(2 , 2)$; this is the knot $16-223382 (a)$.
\item[5)] $\rm F_1 = C(1 , 2)$ , $\rm F_2 = C(3)$: By adding one supplementary crossing in the central band connecting $\rm F_1$ and $\rm F_2$, we get the knot $16-220003 (a)$.
\end{itemize}

So there are six DH-knots with $c \leq 16$. These knots were listed by Alexander Stoimenow in \cite{stoi} as knots for which no projection is known to be achiral. The method we present provides for each of these six knots a non-minimal achiral projection.
\begin{figure}[h]
\includegraphics[scale=0.5]{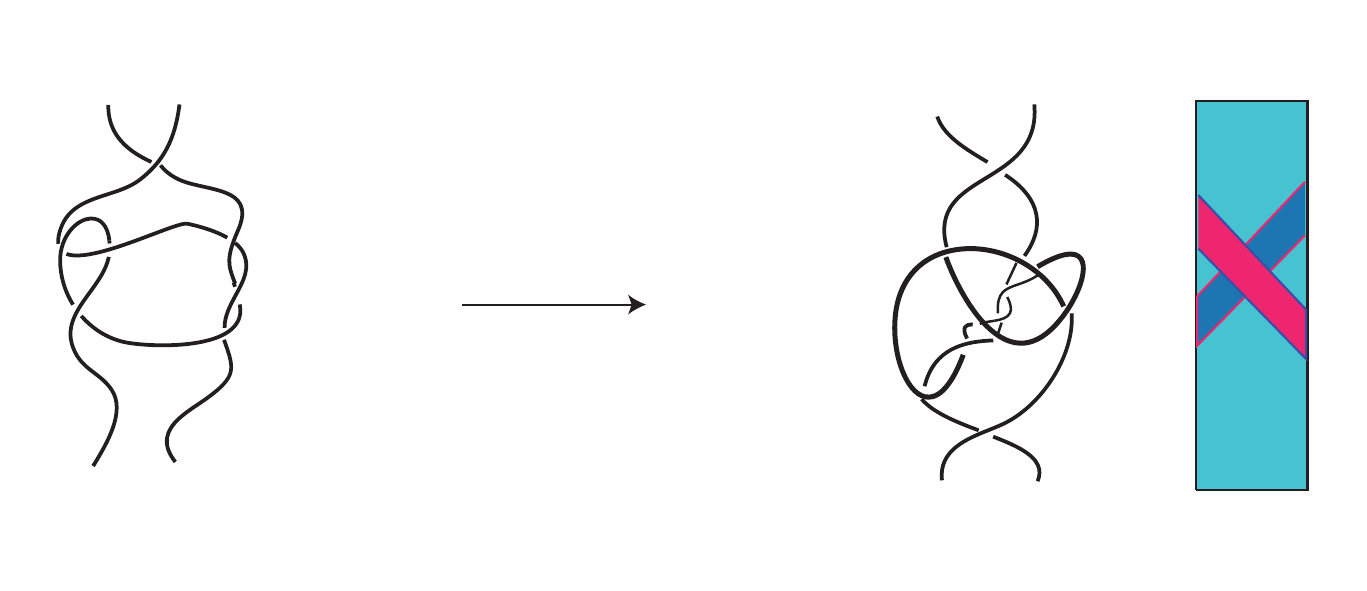} 
\caption{An example of a Dasbach-Hougardy tangle under a symmetrized form.}
\end{figure}
Fig. 25 illustrates our procedure applied to the original Dasbach-Hougardy knot with a cross-plumbing of 3 twisted bands. 

It is clear that there are  DH-knots for  every even crossing number $\geq 14$. Moreover if one requires that the tangles $F_i$ be rational, an exhaustive list can be obtained.
Alexander Stoimenow also listed in \cite{stoi} four +achiral knots for which no achiral projection is known. However they are not alternating and therefore our method cannot be used. Apparently, the existence pf an achiral projection for these knots is still unknown.
By Knotscape, these knots are only +achiral with order equal to 4.

\begin{definition}
Let $K$ be a +AAA knot. The depth of $K$ is the depth of the primary tangle of a minimal projection of $K$.
\end{definition} 

Here is a small analysis of knots with small depth.

{\it Depth 0.}  The +AAA knots with depth 0 are rational knots. The structure tree for any rational knot is an interval with $N$ vertices. The sign of the weights alternates along the interval. If the knot is achiral, the automorphism $\widetilde{\Phi}$ of the tree is the reflection through the middle of the interval. So $N = 2k$. The primary tangle is one half of the interval and  properties of the +achirality automorphism of the structure tree implies that the weights satisfy the equality $a_i = a_{2k-i+1}$ for $1 \leq i \leq k$.

On the other hand, rational knots are generally classified  by an integer $p \geq 2$ and an integer $q$ prime to $p$ such that $1 \leq q \leq (p-1)$. The existence of the equality
$a_i = a_{2k-i+1}$ for $1 \leq i \leq k$ is equivalent to $q^2 \equiv -1~\bmod p$. This condition on the integers $p$ and $q$, is a classical condition necessary and sufficient for a rational knot to be +achiral.

In conclusion, we see that Key Theorem 3.1 enables us to recover without difficulty this well-known condition of the achirality of rational knots. As a bonus, we obtain that the order of +achirality is equal to 4.

Achirality is generally not visible on a minimal projection. The condition for this is that the weights are even. A projection that does not satisfy this condition can be changed to a projection with even weights by the systematic use of gimmicks, thereby adding new crossings.

\begin{definition}
A {\bf Montesinos tangle} is an arborescent tangle with the principal decomposition $\langle {\rm F ;( \rm F_1, \dots, \rm F}_{n});0 \rangle$ where every tangle $\rm F_i$ for $i=1, \dots ,n$, is a rational tangle.
\end{definition}
{\it Depth 1.} The +AAA knots of depth 1 have a Montesinos tangle for the primary tangle. But these knots are not Montesinos knots.

Montesinos knots are never achiral. For those which are alternating, one can easily prove this fact since the structure tree of Montesinos knots does not admit automorphism which satisfies the conditions of achirality.

{\it Depth 2.} Typical examples are Dasbach-Hougardy knots. The smallest depth that may require the use of the $\alpha$-move is depth 2.

\section{Order 4 Theorem 7.1}
\subsection{Statement and proof of order 4 Theorem 7.1}

\begin{theorem}[Order 4 Theorem]
Let $K$ be an alternating +achiral knot without minimal  $+$achiral projection. Then the order of +achirality of $K$ is equal to 4.
\end{theorem}

Note that we do not assume that $K$ is arborescent. The theorem is valid in the class of +achiral alternating knots.

\begin{proof}

If a Haseman circle is invariant, there is nothing to prove since in this case the order of + achirality is equal to 4. See \S 4.

Suppose therefore that there is a jewel $D$ which is invariant. We perform on $\Pi$ the construction we have already done in \cite{erquwe} called the Filling Construction. For the convenience of the reader, we remind you. 

We start with a $\pm$achiral knot and by a minimal projection $\Pi$ of $K$ on $S^2$ with a jewel $D$ invariant by the automorphism $\widetilde{\Phi}$  which acts on the structure tree ${\cal A}(K)$. Let $\gamma_1 , \dots , \gamma_k$ be the Haseman circles which are the boundary components of $D$. Each $\gamma_i$ bounds in $S^2$ a disc $\Delta_i$ which does not meet the interior of $D$. The projection $\Pi$ intersects $\gamma_i$ at four points. Inside $\Delta_i$, the projection $\Pi$ joins either opposite points or adjacent points. In the first case, we replace $\Pi \cap \Delta_i$ by a singleton and in the second case by a 2-spire (see definition in \S8) appropriately placed. We thus obtain a projection $\Pi^*$ of a new knot $K^*$. With an appropriate choice of crossings, we obtain a minimal alternating projection $\Pi^*$  where flypes cannot take place.

Suppose that $K$ is +achiral. Then $K^*$ inherits the property of +achirality of $K$. We are therefore exactly in the situation dealt with in \S 3.2 concerning the interpretation of Key Theorem 3.1. The +achirality is visible on $\Pi^*$. There is a finite order rotation  acting on $S^2$ which leaves $D$ invariant and this rotation followed by the reflection through $S^2$ realises the +achirality of $K^*$.

We now return to the knot $K$. The same twisted rotation with the flypes realizes the +achirality of $K$. We know from Corollary 2.1 that the order, say $n$  of the rotation is equal to the order of the twisted rotation. So we just have to focus on the rotation. As such, it has two fixed points in $S^2$. Away from these fixed points the action is free.
Let us first observe that the fixed points cannot be on $\Pi$ since this implies that we would be dealing with $-$achirality. There are now two possibilities. 

\subsubsection*{First possibility} The two fixed points are in $D$. Then, the finite order rotation freely permutes the boundary components of $D$. There are no short orbits. We can modify the projection inside the discs $\Delta_i$ to obtain a new minimal projection invariant by the twisted rotation. 

\subsubsection*{Second possibility} The two fixed points are outside of $D$, and therefore are in two discs $\Delta_i$ and $\Delta_j$ which are swapped by reflection. Consider one disc, say $\Delta_i$, with its boundary $\gamma_i$.  Along $\gamma_i$, in an enough small neighbourhood, there are 4 arcs of $\Pi$, cutting transversely $\gamma_i$, and nothing else of $\Pi$. The checkerboard surface  of $\Pi$ implies that the black and white colours in the regions determined by $\Pi$ alternate as we move along $\gamma_i$. Since $\Pi$ is a projection of an achiral knot, the colours are exchanged by  symmetry. So the order of the rotation is equal to 4. Since the order of the twisted rotation is equal to the order of the rotation, the order of +achirality is equal to 4.

We conclude by observing that the two discs $\Delta_i$ and $\Delta_j$ are exchanged by reflection and form a short orbit.

\medskip
\noindent{\bf Summary.}  There are two cases where short orbits appear. Both have the order of +achirality equal to 4. One case is that a Haseman circle is invariant. The second case is when a jewel is invariant and when the centres of rotation are outside the jewel, as we have just seen. In these two cases it is easy to find examples where the +achirality is not visible on a minimal projection. In all other cases, there is no short orbit, for instance when the order of +achirality is not equal to 4. There is therefore always a minimal $+$achiral projection for them.
\end{proof}

\subsection {Kauffman-Jablan Conjecture}

Let us recall some relevant facts about the checkerboard graphs $G(\Pi)$ and $G^{\star}(\Pi)$ dual to each other. Suppose $K$ is an alternating knot. If $K$ has a minimal achiral projection $\Pi$, the graphs $G(\Pi)$ and $G^{\star}(\Pi)$ are isomorphic (see Proposition 7.4 in \cite{erquwe}). Recall that the two graphs are isomorphic for a certain minimal projection if the knot is $-$achiral since these knots satisfy Tait's conjecture. However in the case of $+$achirality,  the Dasbach-Hougardy knot is an example which has no minimal projection $\Pi$ but $G(\Pi)$ is isomorphic to $G^{\star}(\Pi)$ (see \cite{daho}). 

The Dasbach-Hougardy knot being arborescent, it motivates the Kauffman-Jablan Conjecture \cite{kaja} which we state here in the following form:

\medskip
\noindent {\bf Kauffman-Jablan Conjecture}.
{\it Let $K$ be an alternating knot which is +achiral but not $-$achiral. If $K$ has no minimal achiral projection, then $K$ is arborescent.}
\medskip

The knot $K_2$ shown in Fig. 26 is a counter-example to the conjecture. See details in Theorem 7.3.

\begin{remark}
In \cite{erquwe} Section 7, we incorrectly announced that there are knots which are counter-examples to the conjecture for each order $2^{\lambda}$ and $\lambda \geq 2$. 

In fact, the knots which are counter-examples to the Kauffman-Jablan conjecture only exist for order 4. 
\end{remark}

\subsection {+Achirality for every order $2^\lambda$}

Using the symmetry of the jewel,  constructive proof can be given but not presented here for the following result:

\begin{theorem}
For each $\lambda \!\geq\! 1$, there is an alternating (non-arborescent) +achiral knot 
$K_  \lambda$ such that:
\begin{enumerate}
\item the order of +achirality of $K_  \lambda$ is equal to $2 ^ \lambda$;
\item $K_  \lambda$ is not $-$achiral.
\item there exists a minimal  $+$achiral alternating projection of $K_\lambda$.
\end{enumerate}
\end{theorem}

\begin{figure}[h]
\includegraphics[scale=0.17]{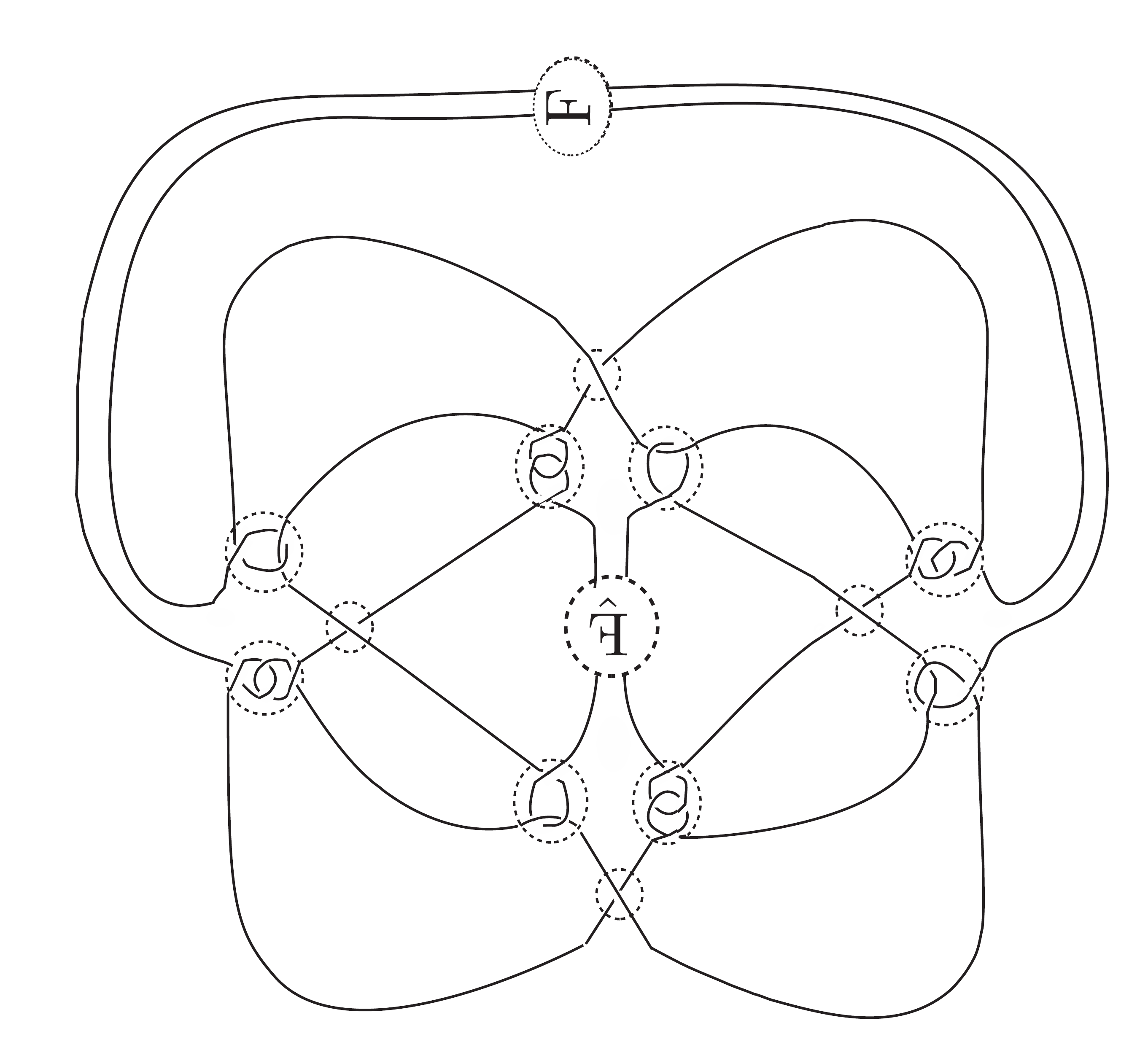}
\caption{The knot $K_2$.}
\end{figure}

\begin{figure}[h!]
\includegraphics[scale=.33]{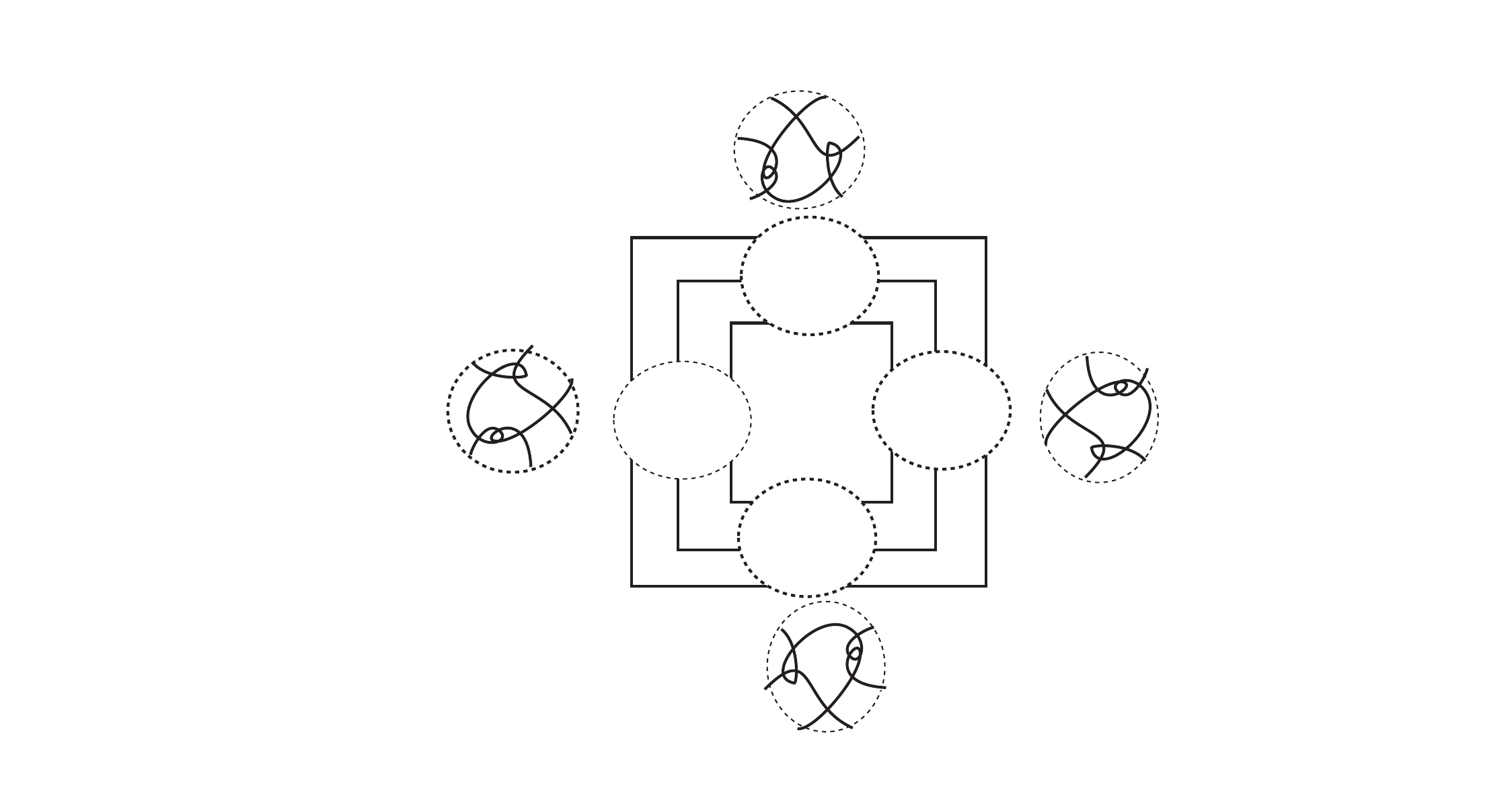}
\caption{The knot $K_1$.}
\end{figure}

The $+$achirality can also be not visible on a minimal projection with non- arborescent knots.

\begin{theorem}
There is an alternating +achiral knot $L_2$ which is not $-$achiral and is such that:
\begin{enumerate}
\item  the order of +achirality of $L_2$ is equal to  $4$,
\item the knot $L_2$ is not arborescent.
\item there is  no achiral minimal  projection of $L_2$.
\end{enumerate}
\end{theorem}

\begin{proof}

First, consider the alternating knot $K_2$ shown in Fig. 26.

\medskip
\noindent{\bf Claim.} $K_2$ satisfies Theorem  7.1. 
Define $K_1$ the knot from $K_2$ by ``deleting" the tangles  $F$ and $\widehat {F}$. Then $K_1$ is a polyhedral alternating knot with four 6-tangles as described in Fig. 27 (a {\bf 6-tangle} is a pair $(\Sigma,\Pi \cap \Sigma )$ where $\Sigma$ is a disc and $\Pi$ is the knot projection such that  $\Pi \cap  \partial\Sigma$ is $6$ points).

It is easy to realise that $K_1$ is $+$achiral  with an homeomorphism $g$ of order 4. Let us denote one of these four 6-tangles by $G$. Since $G \nsim G^v$, $K_1$ is not $-$achiral.

Let us return to the knot $K_2$. For the tangles $\rm F$ and $\rm F^*$  of $K_2$, $g$ acts as follows: $g(\rm F) = \widehat{\rm F}$ and $g^2(\rm F) = {\rm F^*}$. So $g$ induces a $+$achirality of $K_2$ if and only if $\rm F$ satisfies the $*$-condition.
The choice of $\rm F$ under the conditions
\begin{enumerate}
\item
 $ \rm F \sim  \rm F^*$ and
\item
there is no minimal projection $ \rm F' \sim F$ such that $\rm F' \equiv \rm F'^*$.
\end{enumerate}
gives rise to a knot $K_2$ which illustrates Theorem 7.3.
The simplest tangle that satisfies the above conditions is $P(1, 2)$.
Then the knot $K_2$ is non-arborescent, +achiral of order 4 without minimal $+$achiral projection.
\end{proof}

\section{Appendix: Canonical decomposition of a projection}

In the first three subsections, we do not assume that the link projections are alternating.

\subsection{Diagrams}

\begin{definition}
A {\bf planar surface} $\Sigma$ is a compact connected surface embedded in the 2-sphere $S^2$. We denote by $k+1$ the number of connected components of the boundary $\partial \Sigma$ of $\Sigma$. 
\end{definition}
We consider finite graphs $\Gamma$ embedded in $\Sigma$ which satisfy the following four conditions:
\begin{enumerate}
\item vertices of $\Gamma$ have valency 1 or 4,
\item let $\partial\Gamma$ be the set of vertices of $\Gamma$ of valency 1. Then $\Gamma$ is properly embedded in $\Sigma$, i.e. $\partial \Sigma \cap \Gamma = \partial \Gamma$,
\item the number of vertices of $\Gamma$ contained in each connected component of $\partial \Sigma$ is equal to 4,
\item a vertex of $\Gamma$ of valency 4 is called a {\bf crossing}. We require that at each crossing an upper thread and a lower thread be chosen. We denote by $c$ the number of crossings.
\end{enumerate}

\begin{definition}
The pair $D = (\Sigma , \Gamma)$ is called a {\bf diagram}.
\end{definition}

\begin{definition}
 A {\bf singleton} is a diagram homeomorphic to Fig. 28. 
\end{definition}

\begin{figure}[h]
\includegraphics[scale=0.20]{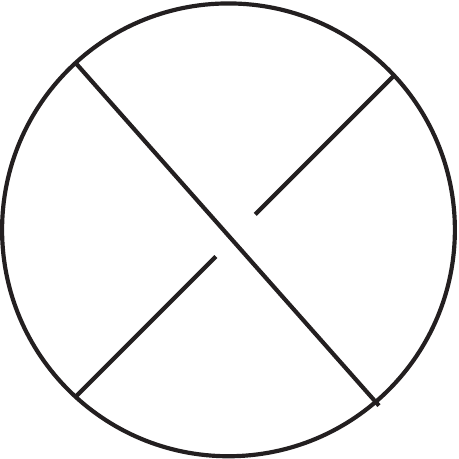}
\caption{A singleton.}
\end{figure}

The (signed) weight  of a crossing on a band is defined according to Fig. 29.

\begin{figure}[h]
\includegraphics[scale=0.40]{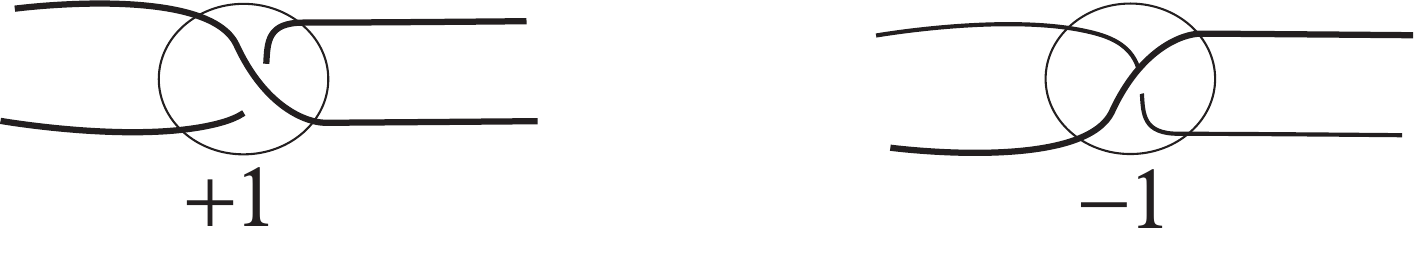}
\caption{The weight of a crossing on a band.}
\end{figure}
{\it First hypothesis.} Crossings along the same band have the same sign. In other words we assume that a Reidemeister move of type II cannot be applied to reduce the number of crossings along a band.

\begin{figure}
\includegraphics[scale=0.30]{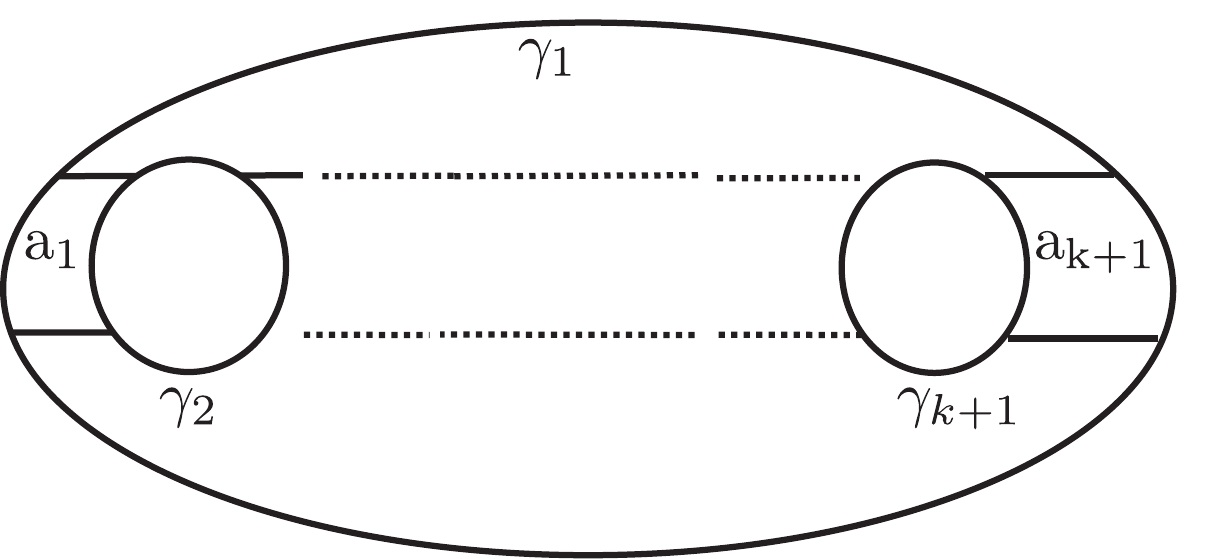}
\caption{A twisted band diagram.}
\end{figure}

\begin{definition}
 A {\bf  twisted band diagram} is a diagram homeomorphic to Fig. 30.
\end{definition}

In  Fig. 30 the boundary components of $\Sigma$ are denoted by $\gamma_1 , \dots , \gamma_{k+1}$ where $k+1 \geq 1$.  $\vert a_i \vert$ denotes the number of crossings between $\gamma_{i-1}$ and $\gamma_i$. The sign of $a_i$ is the sign of the crossings. The integer $a_i$ will be called an {\bf intermediate weight}.

If $k+1 = 1$, the twisted band diagram $ (\Sigma , \Gamma)$ (i.e $\Sigma$ is a disc) is called a {\bf spire}. 

If $k+1 = 2$, the twisted band diagram is a {\bf twisted annulus}. 

{\it Second hypothesis.} 
For a spire, we require that $\vert a_1 \vert \geq 2$. 
For a twisted annulus, we require that $a_1+a_2 \neq 0$.

\begin{remark} Using flypes and Reidemeister II moves, we can reduce the number of crossings of a twisted band diagram in such a way that either $a_i \geq 0$ for all $i = 1 , \dots , k+1$ or $a_i \leq 0$ for all $i = 1 , \dots , k+1$.
This reduction process is not quite canonical, but any two diagrams thus reduced are equivalent by flypes. This is sufficient for our purposes.
\end{remark}

{\it Third hypothesis.} We assume that in a twisted band diagram, all non-zero $a_i$  have the same sign.

{\bf Notation.} The sum of the $a_i$ is called the {\bf weight} of the twisted band diagram and is denoted by $a$. If $k+1\geq 3$, we can have $a = 0$.

\subsection{Haseman circles}
\begin{definition}
 If $k+1\geq 3$ we can have $a = 0$ if:
\begin{itemize}
\item[i)] $\gamma$ bounds a disc $\Delta$ in $\Sigma$.
\item[ii)] There exists a properly embedded arc $\alpha \subset \Delta$ such that $\alpha \cap \Gamma = \emptyset$ and $\alpha$ is not boundary parallel. The arc $\alpha$ is called a {\bf compressing arc} for $\gamma$.
\end{itemize}
\end{definition}

{\it Fourth hypothesis.} Haseman circles are incompressible.

Two Haseman circles are said to be {\bf parallel} if they bound an annulus $A \subset \Sigma$ such that the pair $(A , A \cap \Gamma)$ is diffeomorphic to Fig. 31.

\begin{figure}[h]
\includegraphics[scale=0.15]{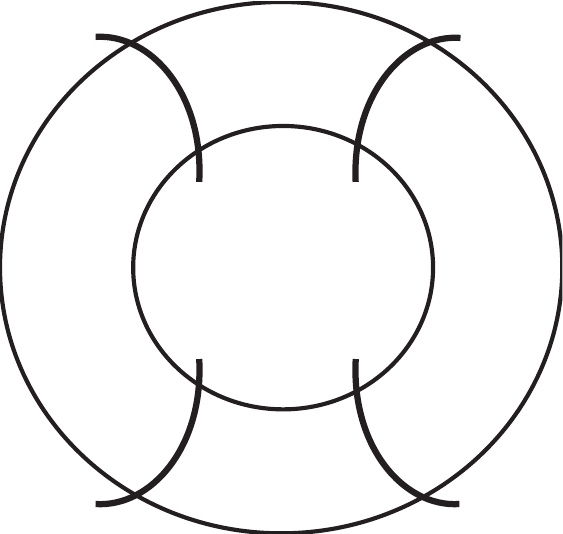}
\caption{Parallel Haseman circles.}
\end{figure}

Analogously, we define a Haseman circle $\gamma$ to be {\bf boundary parallel} if there exists an annulus $A \subset \Sigma$ such that:
\begin{itemize}
\item[1)] the boundary $\partial A$ of $A$ is the disjoint union of $\gamma$ and a boundary component of $\Sigma$;
\item[2)] $(A , A \cap \Gamma)$ is diffeomorphic to Fig. 31.
\end{itemize}
\begin{figure}
\includegraphics[scale=0.3]{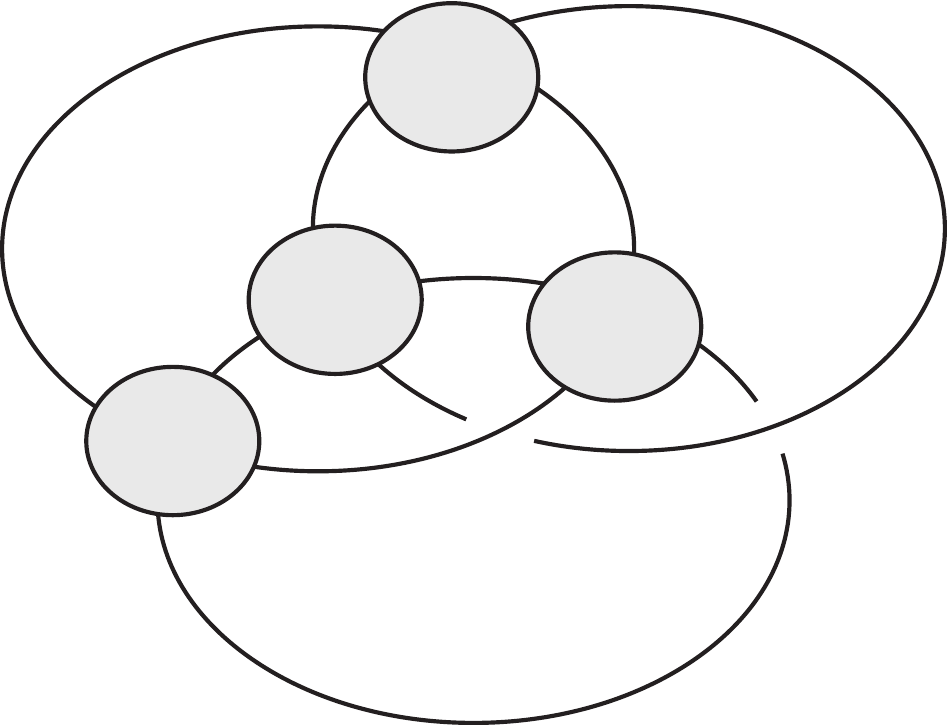}
\caption{A jewel.}
\end{figure}

\begin{definition}
 A {\bf jewel} is a diagram that satisfies the following four conditions:
\begin{itemize}
\item[a)] it is not a singleton.
\item[b)] it is not a twisted band diagram with $k+1 = 2$ and $a = \pm 1$.
\item[c)] it is not a twisted band diagram with $k+1 = 3$ and $a = 0$.
\item[d)] every Haseman circle in $\Sigma$ is either boundary parallel or bounds a singleton.
\end{itemize}
\end{definition}

\noindent {\bf Comments.}
\begin{enumerate}
\item The diagrams listed in a), b) and c) satisfy condition d) but we do not want them to be jewels. Accordingly, a jewel is neither a singleton nor a twisted band diagram. 

\begin{figure}[h]
\includegraphics[scale=0.6]{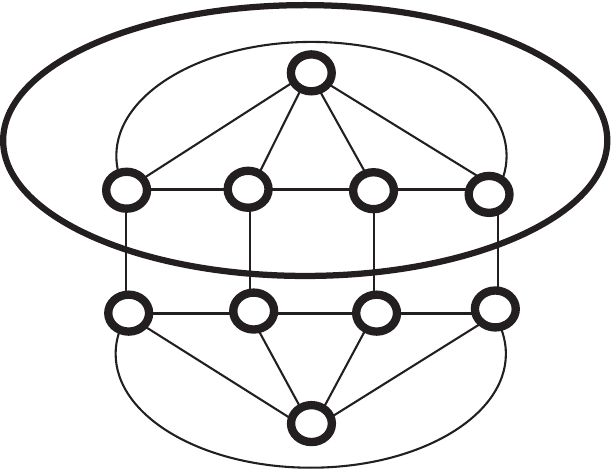}
\caption{10*** is a sum tangle of two 6*.}
\end{figure}

\item

It is necessary to make a comparison between jewels (as defined in Definition 9.7) and the basic Conway polyhedra. Our notion of jewel is more restrictive than the notion of basic polyhedron since a jewel is a diagram with each Haseman circle trivial. For Conway (and others) a basic polyhedron is ``lune-free", where a lune is a portion of a diagram with two edges connecting the same two vertices. In other words, in a basic polyhedron every vertex is connected to four different ones. Therefore a basic polyhedron can be a sum tangle of several jewels. Thus, it can contain non-trivial Haseman circles.  Typically, the basic polyhedron $10^{***}$ is a sum tangle of two $6^*$.
\end{enumerate}

\subsection{Families of Haseman circles of a projection}

\begin{definition}
 A {\bf link projection} $\Pi$ (also called projection for short) is a diagram in $\Sigma = S^2$. 
\end{definition}

{\it Fifth hypothesis.} The projections we consider are connected and prime.

\begin{definition} 
 Let $\Pi$ be a link projection. A {\bf family of Haseman circles} for $\Pi$ is a set of Haseman circles satisfying the following conditions:
\begin{itemize}
\item[1.] any two circles are disjoint.
\item[2.] no two circles are parallel.
\end{itemize}
\end{definition}

Note that a family is always finite, since a projection has a finite number of crossings.

Let $\mathcal{\rm H} = \{ \gamma_1 , \dots , \gamma_n \}$ be a family of Haseman circles for $\Pi$. Let $R$ be the closure of a connected component of $S^2 \setminus  \bigcup_{i=1}^{i=n} \gamma_i$. We call the pair $(R , R \cap \Gamma)$ a {\bf diagram} of $\Pi$ determined by the family $\mathcal{H}$.

\begin{definition}
 A family $\mathcal{C}$ of Haseman circles is an {\bf admissible family} if each diagram determined by it is either a twisted band diagram or a jewel. An admissible family is {\bf minimal} if the deletion of any circle transforms it into a family that is not admissible. 
\end{definition}

The following theorem is the main structure theorem on link projections proved in \cite{quwe}. It is mainly due to Bonahon and Siebenmann. 

\begin{theorem}[Existence and uniqueness theorem of minimal admissible families] Let $\Pi$ be a link projection in $S^2$. Then:
\begin{itemize}
 \item[i)] there exist minimal admissible families for $\Pi$.
\item[ii)] any two minimal admissible families are isotopic, by an isotopy which respects $\Pi$.
\end{itemize}
\end{theorem}

\begin{definition}
The minimal admissible family will be called the {\bf canonical Conway family} for $\Pi$ and denoted by $\mathcal{C}_{can}$. The decomposition of $\Pi$ into twisted band diagrams and jewels determined by $\mathcal{C}_{can}$ will be called the {\bf canonical decomposition} of $\Pi$.
\end{definition}
An element of the canonical Conway family is called a {\bf canonical Conway circle} and it can be of $3$ types:
\begin{itemize}
\item[(1)] a circle that separates two jewels.
\item[(2)] a circle that separates two twisted band diagrams.
\item[(3)] a circle that separates a jewel and a twisted band diagram.
\end{itemize}
\begin{figure}[h!]
\includegraphics[scale=0.6]{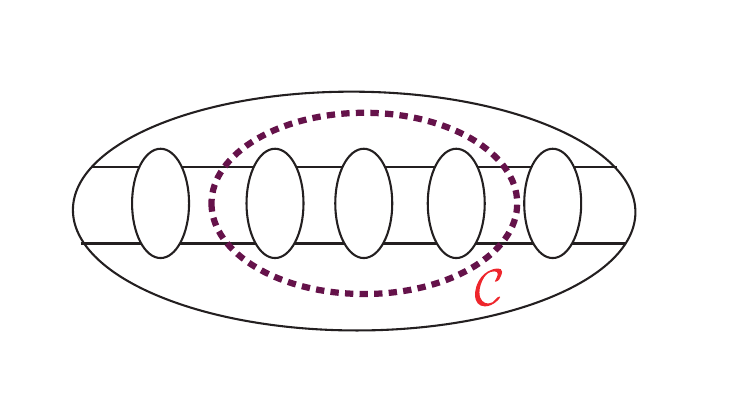}\vspace{-1em}
\caption{$C$ is not a canonical Conway circle.}
\end{figure}
\begin{example} The circle $\cal C$ in Fig. 34 (in dotted lines) is not a canonical Conway circle.
\end{example}

It can happen that $\mathcal{C}_{can}$ is empty. The following proposition tells us when this happens.

\begin{proposition} Let $\Pi$ be a link projection. Then $\mathcal{C}_{can} = \emptyset$ if and only if $\Pi$ is either a jewel with an empty boundary (i.e., $v = 0$) or the minimal projection of the torus knot/link of type $(2 , m)$.
\end{proposition}

\noindent {\bf Comment.} A jewel with an empty boundary is nothing other than a basic polyhedron in J the sense of John Conway which is indecomposable in the sense of the sum tangle. The minimal projection of the torus link of type $(2 , m)$, can be considered as a twisted band diagram with $v = 0$.

\begin{definition}
The {\bf arborescent part} of a graph $\Gamma \subset S^2$ is the union of the twisted band diagrams determined by the canonical Conway family.

The {\bf polyhedral part} of $\Gamma \subset S^2$ is the union of the jewels determined by the canonical Conway family. 

A knot is {\bf arborescent} if it has a projection such that all diagrams determined by the canonical Conway family are twisted band diagrams.
\end{definition}

\begin{remark} The adjective ``arborescent" (or equivalently ``algebraic") has several meanings in the literature. We have adopted the more restrictive one based on 2-dimensional diagrams. Due to their 3-dimensional perspective, Bonahon-Siebenmann have a more permissive definition. For example, Conway showed that some diagrams based on his basic polyhedron (= jewel) $6^*$, which are polyhedral in our sense, can be turned into algebraic diagrams in his sense by adding some more crossings. Typically the knot $10_{99}$ in Rolfsen's notations is such a knot. Note that this knot is +achiral, and also $-$achiral. The two symmetries can easily be seen on a $6^*$ projection.
\end{remark}

\subsection{Structure tree  ${\cal{A}} (K)$}

Now we assume that knots and links are alternating.

\medskip
\noindent {\bf Construction of ${\cal{A}} (K)$.}
Let $K$ be an alternating link and let $\Pi$ be a minimal projection of $K$. Let $\mathcal{C}_{can}$ be the canonical Conway family for $\Pi$. We construct the tree ${\cal{A}} (K)$ as follows. Its vertices are in bijection with the diagrams determined by $\mathcal{C}_{can}$. Its edges are in bijection with the Haseman circles of $\mathcal{C}_{can}$. The extremities of an edge (associated to a Haseman circle $\gamma$) are the vertices which represent the two diagrams containing $\gamma$ in their boundary. Since the diagrams are planar surfaces of a decomposition of the 2-sphere $S^2$ and that $S^2$ has genus zero, the constructed graph is a tree. This tree is ``abstract", i.e., it is not embedded in the plane. 

We have two kinds of vertices: $B-vertices$ and $J-vertices$ of  ${\cal{A}} (K)$; if the vertex represents a twisted band diagram, we label it with the letter $B$ and by the weight $a$ and if the vertex represents a jewel, we label it with the letter $J$. 

\begin{remark}\ 
\begin{enumerate}
\item The tree ${\cal{A}} (K)$ is independent of the minimal projection chosen to represent $K$. This is an immediate consequence of the Flyping Theorem. Indeed, as we have seen,  the flypes  modify the decomposition of the weight $a$ of a twisted band diagram as sum of intermediate weights, but the sum remains constant. A flype also modifies  the way diagrams are embedded in $S^2$. Since the tree is abstract, a flype has no effect on it (\cite{quwe} ,Section 6). This is why we call it the {\bf structure tree} of $K$ (and not of $\Pi$). 

\item  ${\cal{A}} (K)$ contains some informations about the decomposition of $S^2$ into diagrams determined by $\mathcal{C}_{can}$ but we cannot reconstruct the decomposition from it. However one can do better if no jewels are present. In this case the link (and its minimal projections) are called {\bf arborescent} by Bonahon-Siebenmann. They produce a planar tree that actually encodes a given arborescent projection. See \cite{bosi} for details.

\item  We do not encode the orientation of $K$ in ${\cal{A}} (K)$.
\end{enumerate}
\end{remark}

\section{Summary}

We summarize some of the main results proved in the paper which are related to the visibility of the +achirality of alternating knots according to the order of +achiraity which is equal to $2^{\lambda}$ with $\lambda \geq 1$.
For any exponent $\lambda \geq 1$, there exist +achiral knots (see \S 7.3). 

\renewcommand{\arraystretch}{1.8}
\begin{center}
\begin{tabular}{ |c|c|c| }

\hline

Order of $+$achirality $2^\lambda$ & Minimal achiral projection  & Achiral projection \\ 


\hline

$\lambda=2$ arborescent & No in general & Yes (Theorem 5.2) \\

\hline

$\lambda=2$ non-arborescent & \multicolumn{2}{ |c| }{The status in not known in the general case} \\ 

\hline

$\lambda \neq 2$ & \multicolumn{2}{ |c| }{Yes (Theorem 7.2)} \\ 

\hline

\end{tabular}
\end{center}

\address{Universit\'e de Gen\`eve, Section de Math\'ematiques\\
2-4, rue du Li\`evre, CH-1211 Geneva 64, Switzerland\\
\email{nicola.ermotti@edu.ge.ch}\\
\email{cam.quach@unige.ch}\\
\email{claude.weber@unige.ch}\\

\end{document}